\documentclass[11pt]{amsart}
\usepackage{enumerate}
\usepackage{hyperref}
\usepackage{amsmath,amssymb,amsthm}
\usepackage{tikz,ifthen}
\usepackage{algorithm}%
\usepackage{algpseudocode}
\usepackage[T1]{fontenc}
\usepackage[letterpaper, margin=1in]{geometry}
\usepackage{todonotes}
\usetikzlibrary{decorations.markings}

\hypersetup{
    pdftitle = {A half-integral Erd\H{o}s-P\'osa theorem for directed odd cycles},
   pdfauthor=  {}
   }

\usepackage{appendix}
\newcommand\abs[1]{\lvert #1\rvert}
\newtheorem{theorem}{Theorem}[section]
\newtheorem{lemma}[theorem]{Lemma}
\newtheorem*{LEMLARGEPATH}{Lemma \ref{lem:largepathsystem}}
\newtheorem*{LEMODDWALK}{Lemma \ref{lem:oddwalk2}}

\newtheorem{proposition}[theorem]{Proposition}

\newtheorem{claim}{Claim}
\newtheorem{question}{Question}

\newenvironment{pfofclaim}{\noindent \textsc{Proof of the Claim:}}{\hfill$\Diamond$\medskip}

\usetikzlibrary{decorations.pathmorphing}
\usetikzlibrary{decorations.markings}
\usepackage{caption}
\usepackage{subcaption}

\theoremstyle{definition}

\newcommand\cP{\mathcal{P}}
\newcommand\cQ{\mathcal{Q}}
\newcommand\cR{\mathcal{R}}
\newcommand\cU{\mathcal{U}}

\newcommand\cA{\mathcal{A}}

\newcommand\cB{\mathcal{B}}

\newcommand{\clos}{\operatorname{clos}}

\begin{document}
\title[A half-integral Erd\H{o}s-P\'osa theorem for directed odd cycles]{A half-integral Erd\H{o}s-P\'osa theorem for \\directed odd cycles}
\author{Ken-ichi Kawarabayashi}
\author{Stephan Kreutzer}
\author{O-joung Kwon}
\author{Qiqin Xie}
\address[Kawarabayashi]{National Institute of Informatics, 2-1-2, Hitotsubashi, Chiyoda-ku, Tokyo, Japan}
\address[Kreutzer]{Logic and Semantics, TU Berlin, Berlin, Germany}
\address[Kwon]{Department of Mathematics, Hanyang University, Seoul, South Korea}
\address[Kwon]{Discrete Mathematics Group, Institute~for~Basic~Science~(IBS), Daejeon,~South~Korea}
\address[Xie]{Department of Mathematics, College of Sciences, Shanghai University, 99 Shangda Rd., Shanghai, China, 200444}
\email{{k\_keniti@nii.ac.jp}}
\email{{stephan.kreutzer@tu-berlin.de}}
\email{{ojoungkwon@hanyang.ac.kr}}
\email{{qqxie@shu.edu.cn}}
\thanks{Corresponding author: O-joung Kwon. An extended abstract of this paper appeared in the proceedings of SODA2023~\cite{KawaKKX2023}.
The first author is supported by JSPS Kakenhi Grant Number JP18H0529 and 22H05001. The second author is supported by the European Research Council (ERC) under the
    European Union's Horizon 2020 research and innovation programme
    (ERC consolidator grant DISTRUCT, agreement No.\ 648527). The third author is supported by the National Research Foundation of Korea (NRF) grant funded by the Ministry of Education (No. NRF-2021K2A9A2A11101617 and No. RS-2023-00211670) and supported by the Institute for Basic Science (IBS-R029-C1). The fourth author is supported by National Natural Science Foundation of China (No. 12201390) and the National Key R\&D Program of China (No. 2022YFA1006400). Declarations of interest: none. }
\date{\today}
\begin{abstract}
We prove that there exists a function $f:\mathbb{N}\rightarrow \mathbb{R}$ such that every directed graph $G$ contains either $k$ directed odd cycles where every vertex of $G$ is contained in at most two of them, or a set of at most $f(k)$ vertices meeting all directed odd cycles. We give a polynomial-time algorithm for fixed $k$ which outputs one of the two outcomes. 
This extends the half-integral Erd\H{o}s-P\'osa theorem for undirected odd cycles by Reed [Combinatorica 1999] to directed graphs.

\end{abstract}
\maketitle

\section{Introduction}
Erd\H{o}s and P\'osa~\cite{ErdosP1965} proved that for every undirected graph $G$ and every positive integer $k$, $G$ either contains  $k$ pairwise vertex-disjoint cycles or a set of $\mathcal{O}(k\log k)$ vertices that meet all cycles of~$G$. This result has been extended to cycles satisfying various constraints: long cycles~\cite{RobertsonS1986, BirmeleBR2007, FioriniH2014, MoussetNSW2016, BruhnJS2017}, cycles with modularity constraints~\cite{Thomassen1988, HuyneJW2017, BatenHJR2020}, cycles intersecting a prescribed vertex set~\cite{KakimuraKM2011, PontecorviW2012, BruhnJS2017, HuyneJW2017}, and holes~\cite{KimK2020}.  We refer to a survey of Raymond and Thilikos~\cite{RaymondT16} for more examples. 
 On the other hand, such a duality does not exist for odd cycles: Lov\'{a}sz and Schrijver~(see~\cite{Thomassen1988}) 
 found a class of graphs, called Escher walls, where they have no two vertex-disjoint odd cycles but there is no constant $c$ such that every Escher wall admits a set of $c$ vertices meeting all odd cycles. Escher walls are illlustrated in Figure~\ref{fig:escher}.

In 1999, Reed~\cite{Reed1999} obtained a half-integral analogue of the Erd\H{o}s-P\'osa theorem for odd cycles, by relaxing the vertex-disjoint packing to a half-integral packing. A family $\mathcal{C}$ of subgraphs in an undirected graph or a directed graph $G$ 
	is a \emph{half-integral packing} if every vertex of $G$ is contained in at most two subgraphs of $\mathcal{C}$.
This theorem of Reed has been recently generalized to group-labelled graphs by Huynh, Joos, and Wollan~\cite{HuyneJW2017}, Gollin et al.~\cite{GollinHKKO2021}, and Gollin et al.~\cite{GollinHKOY2022}.

\begin{theorem}[Reed~\cite{Reed1999}]\label{thm:Reed}
   There is a function $g\colon \mathbb{N}\to \mathbb{R}$ such that for every undirected graph $G$ and every positive integer $k$, $G$ contains a half-integral packing of $k$ odd cycles, or a set of at most $g(k)$ vertices meeting all odd cycles.
\end{theorem}

\begin{figure}[t]
  \centering
  \begin{tikzpicture}[scale=0.4]
  \tikzstyle{w}=[circle,draw,fill=black!50,inner sep=0pt,minimum width=4pt]

 \foreach \y in {0, 2, 4}{
	\draw (0, \y+1)--(0, \y+2);
	\draw (2, \y+1)--(2, \y+2);
	\draw (4, \y+1)--(4, \y+2);
	\draw (6, \y+1)--(6, \y+2);
	\draw (8, \y+1)--(8, \y+2);

	\draw (0, \y+1)--(8, \y+1);

	\draw (0, \y+2)--(1, \y+2);
	\draw (8, \y+1)--(9, \y+1);

	\draw (1, \y)--(1, \y+1);
	\draw (3, \y)--(3, \y+1);
	\draw (5, \y)--(5, \y+1);
	\draw (7, \y)--(7, \y+1);
	\draw (9, \y)--(9, \y+1);

	\draw (1, \y)--(9, \y);

}
	
	\draw (1, 6)--(8, 6);
	
	\draw[rounded corners, thick] (1, 6)-- (1, 6.5)--(10, 6.5)--(10, -.5)--(8, -.5)--(8, 0);
	\draw[rounded corners, thick] (3, 6)-- (3, 6.5+.5)--(10+.5, 6.5+.5)--(10+.5, -1)--(6, -.5-.5)--(6, 0);
	\draw[rounded corners, thick] (5, 6)-- (5, 6.5+.5+.5)--(10+.5+.5, 6.5+.5+.5)--(10+1, -1.5)--(4, -1.5)--(4, 0);
	\draw[rounded corners, thick] (7, 6)-- (7, 6.5+.5+.5+.5)--(10+.5+.5+.5, 6.5+.5+.5+.5)--(10+1.5, -2)--(2, -2)--(2, 0);
   \end{tikzpicture} \caption{An Escher wall, where the middle wall $W$ is bipartite and each thick path $P$ links from one vertex of the top row to the opposite vertex of the bottom row in the middle wall so that  the union of $W$ and $P$ has an odd cycle.}\label{fig:escher}
\end{figure}

\begin{figure}
  \centering
  \begin{tikzpicture}[scale=0.5, decoration={
    markings,
    mark=at position 0.5 with {\arrow{>}}}]
  \tikzstyle{w}=[circle,draw,fill=black!50,inner sep=0pt,minimum width=3pt]

 \foreach \y in {0, 1, 2, 3, 4, 5}{
	\draw[postaction={decorate}] (2, \y)-- (2, \y+1);
	\draw[postaction={decorate}] (4, \y)--(4, \y+1);
	\draw[postaction={decorate}] (6, \y)--(6, \y+1);
	\draw[postaction={decorate}] (8, \y)--(8, \y+1);

	\draw[postaction={decorate}] (1, \y+1) -- (1, \y);
	\draw[postaction={decorate}] (3, \y+1)--(3, \y);
	\draw[postaction={decorate}] (5, \y+1)--(5, \y);
	\draw[postaction={decorate}] (7, \y+1)--(7, \y);
	\draw[postaction={decorate}] (9, \y+1)--(9, \y);
	
	\foreach \x in {1, 2, 3, 4, 5, 6, 7, 8}{
		\draw[postaction={decorate}] (\x, \y)--(\x+1, \y);
	}
}
	\foreach \x in {1, 2, 3, 4, 5, 6, 7, 8}{
		\draw[postaction={decorate}] (\x, 6)--(\x+1, 6);
	}

	\draw[rounded corners, ->]   (9.5, 3) -- (10, 3)-- (10, -1.5) -- (5, -1.5);
	\draw[rounded corners, ->]   (5, -1.5) -- (0, -1.5)--(0, 3)--(.5, 3);

 \foreach \y in {0, 2, 4, 6}{
	\draw[thick, postaction={decorate}] (2+\y, 6)-- (2+\y,7);
	\draw[thick, postaction={decorate}] (2+\y,7)--(2.5+\y, 7.5);
	\draw[thick, postaction={decorate}] (2.5+\y, 7.5)--(3+\y, 7);
	\draw[thick, postaction={decorate}](3+\y, 7)--(3+\y,6);
	\draw (2+\y,7) node [w] {};
	\draw (2.5+\y,7.5) node [w] {};
	\draw (3+\y,7) node [w] {};
}

   \end{tikzpicture} \caption{A bipartite cylindrical grid with some parity-changing paths on the top. It is not difficult to see that there are no two vertex-disjoint directed odd cycles, but one can increase the minimum size of a hitting set by taking a larger construction. }\label{fig:circular}
\end{figure}

For directed graphs, the situations become much more complicated, and not many results are known. Reed, Robertson, Seymour, and Thomas~\cite{ReedRST1996} showed that an analogue of the Erd\H{o}s-P\'osa theorem holds for directed cycles, which confirms a long-standing conjecture of Younger \cite{younger}.  As an application of the directed grid theorem,  Kawarabayashi and Kreutzer~\cite{KawaK2015} proved that an analogue of the Erd\H{o}s-P\'osa theorem holds for directed cycles of length at least $\ell$ for some fixed $\ell$. 
Amiri et al.~\cite{AmiriKKW2016} further extended so that if $H$ is a strongly connected directed graph such that any $H$-subdivision can be obtained as a subgraph of some cylindrical wall (see Figure~\ref{fig:wall}), then an analogue of the Erd\H{o}s-P\'osa theorem holds for $H$-subdivisions. 
Kakimura and Kawarabayashi~\cite{DBLP:journals/siamdm/KakimuraK12} showed that an analogue of the Erd\H{o}s-P\'osa theorem does not hold for directed cycles meeting a prescribed set $S$ (so called directed $S$-cycles), but a $1/5$-integral analogue of the Erd\H{o}s-P\'osa theorem holds (this result was further improved to a half-integral analogue in \cite{soda13}). 
Masa\v{r}\'{\i}k et al.~\cite{MasarikMPRS2019} considered a half-integral packing of directed cycles, and proved that there is a set of $\mathcal{O}(k^6)$ vertices meeting all directed cycles if a directed graph has no half-integral packing of $k$ directed cycles. 
On the other hand, directed cycles with modularity constraints have not been considered much. Very recently, Gorsky, Kawarabayashi, Kreutzer, and Wiederrecht~\cite{GorskyKKW2024} proved that a $1/4$-integral analogue of the Erd\H{o}s-P\'osa theorem holds for directed even cycles.

The main contribution of this paper is to show that a half-integral analogue of the Erd\H{o}s-P\'osa theorem holds for directed odd cycles.
We construct an example, illustrated in Figure~\ref{fig:circular}, showing that an analogue of the Erd\H{o}s-P\'osa theorem does not hold for directed odd cycles even on planar directed graphs. 
This contrasts with the undirected case; it is known that an analogue of the Erd\H{o}s-P\'osa theorem holds for odd cycles on planar graphs~\cite{Reed1999, FioriniHRV2007, KralSS2012}.

\begin{theorem}\label{thm:main}
There is a function $f:\mathbb{N}\rightarrow \mathbb{R}$ such that for every directed graph $G$ and every positive integer $k$, 
$G$ contains a half-integral packing of $k$ directed odd cycles, or 
a set of at most $f(k)$ vertices meeting all directed odd cycles. For every fixed positive integer $k$, there is a polynomial-time algorithm that given a graph $G$, outputs one of the two outcomes.
\end{theorem}

The obtained function $f$ in Theorem~\ref{thm:main} relies on the function $f_{wall}$ for directed grid theorem~\cite{KawaK2015J} by Kawarabayashi and Kreutzer (see Theorem~\ref{thm:KK2}). 
It satisfies $f(1)=0$ and \[f(i)\le 96f_{wall}^2\left(2^{34}i^6+4i\cdot f(i-1)\right)\] for $i\ge 2$. 
The function $f_{wall}$ in~\cite{KawaK2015J} contains an exponential tower, and so does $f$. 
Apart from using the function for the directed grid theorem, the functions appearing in the other parts are elementary. 
We ask whether there is a polynomial function for Theorem~\ref{thm:main}. To achieve this, we either need to obtain a polynomial function for the directed grid theorem, or find another approach which avoids using the directed grid theorem. See Section~\ref{sec:discussion} for more discussion.

\medskip 
We sketch the proof of Theorem~\ref{thm:main}.

To obtain Erd\H{o}s-P\'osa type results for various graph families in the undirected setting, the grid minor theorem~\cite{RobertsonS1986} has been importantly used, see~\cite{RobertsonS1986, Thomassen1988, Reed1999, HuyneJW2017, GollinHKKO2021} for examples.  For directed graphs, 
Kawarabayashi and Kreutzer~\cite{KawaK2015} obtained the directed grid theorem, which shows that every directed graph of sufficiently large directed tree-width contains a cylindrical grid of large order as a butterfly minor. They observed that if a directed graph contains a cylindrical grid of large order as a butterfly minor, then it contains a cylindrical wall of large order as a subgraph. Similarly, we will use a cylindrical wall of large order, which is depicted in Figure~\ref{fig:wall}.

A set $S$ of vertices in a directed graph $G$ is a \emph{hitting set} for directed odd cycles if $S$ meets all directed odd cycles of $G$.
For a directed graph $G$, we denote by $\nu_2(G)$ the maximum size of a half-integral packing of directed odd cycles in $G$, and denote by $\tau(G)$ the minimum size of a hitting set for directed odd cycles in $G$.
     For each positive integer $k$, we define $\alpha_k$ as the minimum integer such that for every directed graph $G$ with $\nu_2(G)<k$, we have $\tau(G)\le \alpha_k$, if such an integer exists, and otherwise $\alpha_k$ is defined to be $\infty$.  It is sufficient to show that $\alpha_k\neq \infty$ for every positive integer $k$. Clearly, $\alpha_1=0$.
 We will prove it by induction on $k$.

    A set $T$ of vertices in a directed graph $G$ is an \emph{$r$-externally-well-linked set} 
	if for all disjoint sets $A$ and $B$ of vertices in $T$ with $\abs{A}=\abs{B}\ge r$, 
	there is a set of $\abs{A}$ vertex-disjoint paths from $A$ to $B$ in $G-(T\setminus (A\cup B))$ (and also from $B$ to $A$).
    We show in Lemma~\ref{lem:welllinked} that if $\alpha_{k-1}\neq \infty$ and a directed graph $G$ with $\nu_2(G)<k$ has a hitting set $T$ of directed odd cycles with $\abs{T}=\tau(G)$, then $T$ is $2\alpha_{k-1}$-externally-well-linked. 
    So, we can argue that if $\tau(G)$ is sufficiently large, then $G$ has large directed tree-width, and it contains a cylindrical wall of large order by the directed grid theorem. However, for our purpose, we need a special cylindrical wall of large order that cannot be separated from $T$ by removing a small set of vertices.

    Such a result was obtained in~\cite{KawaK2015J} (which is the journal version of \cite{KawaK2015}) for ordinary well-linked sets. 
    A set $X$ of vertices in a directed graph $G$ is a \emph{well-linked set} if for all sets $A$ and $B$ of vertices in $X$ with $\abs{A}=\abs{B}$,  there is a set of $\abs{A}$ vertex-disjoint paths from $A$ to $B$ in $G$ (and also from $B$ to $A$).
    Note that in some other papers, well-linked sets indicate externally-well-linked sets defined here.
    Kawarabayashi and Kreutzer~\cite[Theorem 7.1]{KawaK2015J} showed that 
    if $G$ contains a sufficiently large well-linked set $X$, then it contains a large cylindrical wall of order $w$, such that for every set $F$ of $w$ vertices that are out-degree $2$ or in-degree $2$ in the wall, there are $w$ vertex-disjoint paths from $F$ to $X$ in $G$ and from $X$ to $F$ in $G$.

    To relate the $2\alpha_{k-1}$-externally-well-linked set $T$ to some cylindrical wall, we prove the following lemma in Section~\ref{sec:welllinked} that there is a well-linked set $X$ such that $T$ and $X$ cannot be separated by removing a small set of vertices. It may be useful in the other context.
    \begin{lemma}\label{lem:largepathsystem}
	Let $r$ and $p$ be positive integers with $2p(p+1)\ge r$.
	If a directed graph $G$ contains an $r$-externally-well-linked set $T$ of size at least $12p(p+1)+1$, then 
	there exist a path $P$ in $G$ and $A\subseteq V(P)$ with $\abs{A}=p$ such that 
 \begin{itemize}
     \item $A$ is well-linked, and 
     \item for every subset $Z$ of $T$ of size at least $\abs{T}/2$, there is a set of $p$ vertex-disjoint paths from $A$ to $Z$, and there is a set of $p$ vertex-disjoint paths from $Z$ to $A$.
 \end{itemize} 
	\end{lemma}
    Combining with the directed grid theorem, 
    we obtain a required cylindrical wall $W$ of large order that is not separated from $T$ by removing a small set of vertices.

    We take $k$ vertex-disjoint subwalls of $W$ in a natural way, 
    and we may assume that one of them, say $W'$, has no directed odd cycles. As every wall is strongly connected, we can argue that the underlying undirected graph of $W'$ is bipartite. Let $N$ be a large set of vertices of $W'$ such that they have out-degree $2$ or in-degree $2$ in the wall, and they are in the same part of the bipartition of $W'$.
    We prove the following lemma in Section~\ref{sec:oddwalk}. A directed $X$-walk is a directed walk having at least one edge such that both endvertices are in $X$ and all its internal vertices are not in $X$, and 
    a directed $X$-path is a directed $X$-walk that is a directed path.

\begin{lemma}\label{lem:oddwalk2}
Let $k$ be a positive integer, let $G$ be a directed graph, and let $X\subseteq V(G)$. Then $G$ contains either
\begin{enumerate}
\item
a half-integral packing of $k$ odd cycles,
\item
a half-integral packing of $k$ odd $X$-paths whose endvertices are pairwise disjoint,  or
\item
a set $Y$ of at most $4k-1$ vertices such that $G-Y$ has no odd $X$-walk.
\end{enumerate}
\end{lemma}
    We apply this lemma to the set $X=N$ of $W'$ in $G$.

    In case when there is a small set $Y$ of vertices meeting all directed odd $N$-walks, 
    there is a strong component $H$ of $G-Y$ containing most of the vertices in $T$.
    We can argue that more than half of the columns of $W$ are also contained in~$H$.
    On the other hand, if $H$ has a directed odd cycle, then one can find a directed odd $N$-walk, which is a contradiction.
    So, $Y$ together with $T\setminus V(H)$ gives a hitting set for directed odd cycles, which is small.
    In the case when there are many directed odd $N$-paths,  
    we show in Section~\ref{sec:main} that we can use the bipartite cylindrical wall to find a half-integral packing of $k$ odd cycles, which contradicts the assumption that $\nu_2(G)<k$.  
    
    This part of finding a half-integral packing of $k$ directed odd cycles from many directed odd $N$-paths is technical. The odd $N$-paths may intersect with the other part of the bipartite wall $W'$. So, we first extract parity-breaking paths from given odd $N$-paths that intersect a small portion of the bipartite wall, and then find a large subwall $W''$ that is disjoint from selected parity-breaking paths. By connecting $W''$  and parity-breaking paths using paths in $W'$, we obtain a half-integral packing of directed paths that only attach to the boundary of $W''$. Then, using the internal part of $W''$, we complete each directed path to a directed odd cycle.

\section{Preliminaries}\label{sec:prelim}

Let $\mathbb{N}$ be the set of all positive integers, and $\mathbb{R}$ be the set of all reals.
For an integer $m$, we write $[m]$ for the set of positive integers at most $m$. In this paper, all directed graphs have no multiple edges and loops.
Directed walks, directed paths, and directed cycles are simply called walks, paths, and cycles respectively. 

Let $G$ be a directed graph. We denote by $V(G)$ and $E(G)$ the vertex set and the edge set of $G$, respectively.
If $(v,w)$ is an edge, then $v$ is its \emph{tail} and $w$ is its \emph{head}.
For a set $A$ of vertices in $G$, we denote by $G-A$ the directed graph obtained from $G$ by removing all the vertices in $A$, 
and denote by $G[A]$ the subgraph of $G$ induced by $A$.
For two directed graphs $G$ and $H$, let $G\cup H:=(V(G)\cup V(H), E(G)\cup E(H))$ and $G\cap H:=(V(G)\cap V(H), E(G)\cap E(H))$. For a set~$\mathcal{G}$ of directed graphs, we denote by~${\bigcup \mathcal{G}}$ the union of the directed graphs in~$\mathcal{G}$.  

We say that a directed graph $G$ is \emph{strongly connected} if for any two vertices $v$ and $w$ in $G$, there is a path from $v$ to $w$ in $G$ and there is a path from $w$ to $v$ in $G$.
A \emph{strong component} of $G$ is a maximal subgraph of $G$ that is strongly connected. It is well known that the set of strong components of $G$ can be labelled $G_1, G_2, \ldots, G_t$ such that there is no edge from $G_j$ to $G_i$ if $j> i$. Such an ordering is called an \emph{acyclic ordering} of the strong components of $G$.

For sets $A$ and $B$ of vertices in a directed graph $G$, a path is an \emph{$(A, B)$-path} 
if it starts at $A$ and ends at $B$, and all its internal vertices are not in $A\cup B$. 
For a set $A$ of vertices in $G$, an  \emph{$A$-walk} $P$ is a walk having at least one edge such that both endvertices of $P$ are in $A$ and all its internal vertices are not in $A$.
Note that the two endvertices of an $A$-walk may be the same vertex. An $A$-walk is \emph{closed} if its endvertices are the same. 
An $A$-walk is called an \emph{$A$-path} if it is a path. 

	Let $t$ be a positive integer. A family $(G_i:i\in [m])$ of subgraphs in a directed graph $G$ 
	is a \emph{$(1/t)$-integral packing} if every vertex of $G$ is contained in at most $t$  of $G_1, G_2, \ldots, G_m$. When $t=2$, we say that it is a {\em half-integral packing}.

\subsection{Cylindrical walls}

\begin{figure}
  \centering
  \begin{tikzpicture}[scale=0.6, decoration={
    markings,
    mark=at position 0.5 with {\arrow{>}}}]
  \tikzstyle{v}=[rectangle, draw, solid, fill=black, inner sep=0pt, minimum width=3pt, minimum height=3pt]
  \tikzstyle{w}=[circle, draw, solid, fill=black, inner sep=0pt, minimum width=3pt]

 \foreach \y in {1, 2, 3, 4, 5, 6, 7, 8}{
       \node at (8.6, 8-\y) {$P_{\y}$};
}

\foreach \x in {0, 1, 2, 3, 4, 5, 6, 7}{
    \foreach \y in {0, 1, 2, 3, 4, 5, 6, 7}{
          \draw (\x, \y) node [v] () {};
}
}

	\foreach \x in {0, 2, 4}{
		\foreach \y in {1,3,5,7}{
	\draw[postaction={decorate}] (\x, \y+1)-- (\x, \y);	
	\draw[postaction={decorate}] (\x+1, \y)-- (\x+1, \y-1);	
	
	\draw[postaction={decorate}] (\x, \y)-- (\x+1, \y);	
	\draw[postaction={decorate}] (\x+1, \y)-- (\x+2, \y);

    \draw[postaction={decorate}] (\x+1, \y-1)-- (\x, \y-1);	
    \draw[postaction={decorate}] (\x+2, \y-1)-- (\x+1, \y-1);	
	}
	}
	\foreach \x in {6}{
		\foreach \y in {1,3,5,7}{
		
	\draw[very thick, postaction={decorate}] (\x, \y+1)-- (\x, \y);	
	\draw[very thick, postaction={decorate}] (\x+1, \y)-- (\x+1, \y-1);	
	
	\draw[very thick, postaction={decorate}] (\x, \y)-- (\x+1, \y);	
			
	\draw[very thick, postaction={decorate}] (\x+1, \y-1)-- (\x, \y-1);	
	}
	}
    
	\foreach \x in {1,2,3}{
\draw[rounded corners, ->] (0+2*\x-2, 0)--(0+2*\x-2, -.5*\x)--(-1*\x, -.5*\x)--(-1*\x, 4);
\draw[rounded corners] (-1*\x, 4)--(-1*\x, 8+0.5*\x)--(0+2*\x-2, 8+0.5*\x)--(0+2*\x-2,8);
\node at (-1*\x-.5, 4) {$C_{\x}$};

}

	\foreach \x in {4}{
\draw[very thick, rounded corners, ->] (0+2*\x-2, 0)--(0+2*\x-2, -.5*\x)--(-1*\x, -.5*\x)--(-1*\x, 4);
\draw[very thick, rounded corners] (-1*\x, 4)--(-1*\x, 8+0.5*\x)--(0+2*\x-2, 8+0.5*\x)--(0+2*\x-2,8);
\node at (-1*\x-.5, 4) {$C_{\x}$};

}

   \end{tikzpicture} \caption{A cylindrical wall of order $4$. The cycle $C_4$ is depicted using thick edges. Rectangle vertices denote the nails of the wall.}\label{fig:wall}
\end{figure}

 For an integer $k\geq 2$, a \emph{cylindrical wall} $W$ of order $k$ is a
  directed graph consisting of $k$ pairwise vertex-disjoint cycles $C_1, \dots, C_k$, called \emph{columns}, and a set of $2k$ pairwise
  vertex-disjoint paths $P_1, \dots, P_{2k}$, called \emph{rows}, such that 
    \begin{itemize}
  \item for each $i\in [k]$ and $j\in [2k]$, $C_i\cap P_j$ is a path with at least one edge,
 \item the endvertices of $P_i$ are in $V(C_1)\cup V(C_k)$,
    \item the paths $P_1\cap C_i, \dots, P_{2k}\cap C_i$ appear in this order on each $C_i$ and
  \item for odd $i$, the paths $C_1\cap P_i, \dots, C_k\cap P_i$ appear
    in this order on $P_i$, and for even $i$, $C_k\cap P_i,
    \dots, C_1\cap P_i$ appear
    in this order on $P_i$.
  \end{itemize}
	See Figure~\ref{fig:wall} for an illustration of a cylindrical wall of order $4$.
	An endvertex of $C_i\cap P_j$ for some $i\in [k]$ and $j\in [2k]$ is called a \emph{nail}, and we denote by $N^W$ the set of all nails of $W$. Note that an $N^W$-path in $W$ is a path such that its endvertices are nails, but all the internal vertices are not nails. 
    Observe that $W$ contains exactly $4k^2$ nails.

We will use cylindrical walls that do not contain odd cycles.	
Because of the following fact, the underlying undirected graph of such a wall is bipartite. 

\begin{proposition}[Folklore]\label{prop:bipartite}
Let $D$ be a strongly connected directed graph having no
odd cycle. Then, the underlying undirected graph of $D$ is bipartite.
\end{proposition}

	We say that a cylindrical wall is \emph{bipartite} if its underlying undirected graph is bipartite.

	\subsection{Linkages and separations}
	For a positive integer $t$ and sets $A$ and $B$ of vertices in $G$, 
	a family $(P_i:i\in [m])$ of $(A, B)$-paths in $G$ is a \emph{$(1/t)$-integral linkage of order $m$ from $A$ to $B$}
	if it is a  $(1/t)$-integral packing. When $t=1$, we simply call it a linkage.
	A \emph{separation} of a directed graph $G$ is an ordered pair $(A, B)$ of sets of vertices in $G$ such that
	$A\cup B=V(G)$ and there are
  no edges from $A\setminus B$ to $B\setminus A$.
  The \emph{order} of the separation $(A, B)$ is $\abs{A\cap B}$.

  \begin{theorem}[Menger's theorem~\cite{Menger27}]\label{thm:menger}  
  Let $A$ and $B$ be sets of vertices in a directed graph $G$, and let $k$ be a positive integer.
  Then $G$ contains either a linkage of order $k$ from $A$ to $B$, 
  or a separation $(X,Y)$ of order less than $k$ such that $A\subseteq X$ and $B\subseteq Y$.
  \end{theorem}
		
	 We will use the following observation.
	\begin{lemma}\label{lem:mintegral}
	Let $t$ and $m$ be positive integers, and let $A$ and $B$ be sets of vertices in a directed graph $G$.
	If there is a $(1/t)$-integral linkage $\cP_1$ of order $m$ from $A$ to $B$, 
	then there is a linkage $\cP_2$ of order at least $m/t$ from $A$ to $B$ such that $\bigcup \cP_2$ is a subgraph of $\bigcup \cP_1$.
	\end{lemma}
	\begin{proof}
	We may assume that $G=\bigcup \cP_1$.
	Suppose that there is no linkage of order at least $m/t$ from $A$ to $B$ in $G$. Then by Menger's theorem, 
	there is a separation $(C, D)$ of order less than $m/t$ in $G$ such that $A\subseteq C$ and $B\subseteq D$.
	Now, since $\mathcal{P}_1$ is $(1/t)$-integral, 
	each vertex of $C\cap D$ is contained in at most $t$ paths of $\mathcal{P}_1$. Since every path in $\mathcal{P}_1$ contains a vertex of $C\cap D$, the order of $\mathcal{P}_1$ is at most $(\lceil m/t\rceil-1)t$, which is less than $m$. This contradicts the assumption that $\mathcal{P}_1$ has order $m$.
	\end{proof}

	\subsection{Well-linked sets}
	We will discuss two versions of well-linked sets.
	A set $T$ of vertices in a directed graph $G$ is a \emph{well-linked set} if for all sets $A$ and $B$ of vertices in $T$ with $\abs{A}=\abs{B}$,
	there is a linkage of order $\abs{A}$ from $A$ to $B$ in $G$
	and there is a linkage of order $\abs{A}$ from $B$ to $A$ in $G$. It is known that a directed graph has a large well-linked set if and only if it has large directed tree-width. We refer to Section 9.3 in~\cite{digraphbook}.

We will use the following version of the directed grid theorem.

\begin{theorem}[Kawarabayashi and Kreutzer, Theorem 7.1 of \cite{KawaK2015J}]\label{thm:KK2}
There is a function $f_{wall}:\mathbb{N}\to \mathbb{R}$ such that for every positive integer $w$ and every directed graph $G$, if $G$ contains a well-linked set $A$ of order $f_{wall}(w)$, then it contains a cylindrical wall $W$ of order $w$, such that for every set $F$ of $w$ nails of $W$, there are $w$ vertex-disjoint paths from $F$ to $A$ in $G$ and there are $w$ vertex-disjoint paths from $A$ to $F$ in $G$.
\end{theorem}

	A set $T$ of vertices in a directed graph $G$ is an \emph{$r$-externally-well-linked set} 
	if for all disjoint sets $A$ and $B$ of vertices in $T$ with $\abs{A}=\abs{B}\ge r$, 
	there is a linkage of order $\abs{A}$ from $A$ to $B$ in $G-(T\setminus (A\cup B))$
	and there is a linkage of order $\abs{A}$ from $B$ to $A$ in $G-(T\setminus (A\cup B))$.
    This concept naturally appears in the Erd\H{o}s-P\'osa type results, see~\cite{ReedRST1996} for instance.

    For a positive integer $q$, a set $S$ of vertices in a directed graph $G$ is \emph{$q$-linked} if for every set $X\subseteq V(G)$ with $\abs{X}<q$, there is a unique strong component of $G-X$ that contains more than half of the vertices in $S$.

	We use the following relation between $r$-externally-well-linked sets and $q$-linked sets.
	\begin{lemma}\label{lem:wls}
	Let $q$ and $r$ be positive integers with $q\ge r$.
	Every $r$-externally-well-linked set of order at least $6q-4$ is $q$-linked.	
	\end{lemma}
	\begin{proof}
	Let $T$ be an $r$-externally-well-linked set of size at least $6q-4$. To show that $T$ is $q$-linked, we choose a set $X$ of less than $q$ vertices.
	Let $H_1, H_2, \ldots, H_m$ be the set of all strong components of $G-X$, and assume that it is ordered in an acyclic ordering. Suppose for contradiction that there is no strong component of $G-X$ containing more than half of the vertices in $T$. 
	
	We choose a minimum integer $j$ such that 
	$\bigcup_{i\in [j]}V(H_i)$ contains at least $q$ vertices of $T$. As every strong component of $G-X$ has at most $\abs{T}/2$ vertices of $T$, 
	$\bigcup_{i\in [j]}V(H_i)$ contains at most $(q-1)+\abs{T}/2$ vertices of $T$.
	Thus, 
	$\bigcup_{i\in [m]\setminus [j]}V(H_i)$ contains at least \[\abs{T}-(q-1)-\left(q-1+\frac{\abs{T}}{2}\right)=\frac{\abs{T}}{2}-2(q-1)\ge q\] vertices of $T$.
	It implies that there is a linkage of order $q$ from $T\cap (\bigcup_{i\in [m]\setminus [j]}V(H_i))$ to 
	$T\cap (\bigcup_{i\in [j]}V(H_i))$.
	But all these $q$ paths have to contain a vertex of $X$, which is not possible.
	
	We conclude that $T$ is $q$-linked.
\end{proof}

\section{Lemmas on odd $X$-walks}\label{sec:oddwalk}

In this section, we prove the following lemma, which will be used in the proof of Theorem~\ref{thm:main}.
\begin{LEMODDWALK}
Let $k$ be a positive integer, let $G$ be a directed graph, and let $X\subseteq V(G)$. Then $G$ contains either
\begin{enumerate}
\item
a half-integral packing of $k$ odd cycles,
\item
a half-integral packing of $k$ odd $X$-paths whose endvertices are pairwise disjoint,  or
\item
a set $Y$ of at most $4k-1$ vertices such that $G-Y$ has no odd $X$-walk.
\end{enumerate}
\end{LEMODDWALK}

As a first step, we prove the following.

\begin{lemma}\label{lem:oddwalk1}
Let $\ell$ be a positive integer, let $G$ be a directed graph, and let $X\subseteq V(G)$. Then $G$ contains either
\begin{enumerate}
\item
a set of $\ell$ odd $X$-walks such that every vertex of $G$ is used in at most two of them including the number of repetitions in each walk,  or
\item
a set $Y$ of at most $\ell-1$ vertices such that $G-Y$ has no odd $X$-walk.
\end{enumerate}
\end{lemma}
\begin{proof} 
We obtain a new directed graph from $G$ by splitting each vertex $v$ into two vertices $v_1$ and $v_2$, 
and adding edges $(v_1, w_2), (v_2, w_1)$  if $(v,w)$ is an edge of $G$.
Formally, let $D$ be the bipartite directed graph with bipartition $(A, B)$ such that 
\begin{itemize}
    \item $A=\{v_1: v\in V(G)\}$ and $B=\{v_2: v\in V(G)\}$, and
    \item $E(D)=\{(v_1, w_2), (v_2, w_1):(v,w)\in E(G)\}$.
\end{itemize}
Let $X_A:=\{v_1:v\in X\}$ and $X_B:=\{v_2:v\in X\}$. For a vertex $v_i\in V(D)$, we say that $v$ is the \emph{original vertex} of $v_i$.

Observe that $X_A$ and $X_B$ lie in distinct parts of $D$, and therefore, any path from $X_A$ to $X_B$ in $D$ has odd length.

Assume that there is a family $\cQ$ of $\ell$ vertex-disjoint paths from $X_A$ to $X_B$ in $D$. We obtain from each path $Q\in \cQ$, a walk $Q^*$ in $G$ by taking the sequence of corresponding original vertices. Then $(Q^*:Q\in \cQ)$ is a family of $\ell$ odd $X$-walks in $G$ such that every vertex of $G$ is used in at most two of them including the number of repetitions in each walk.
In this case, we get the first conclusion.
Otherwise, by Menger's theorem, there is a separation $(S, T)$ in $D$ of order at most $\ell-1$ such that
$X_A\subseteq S$ and $X_B\subseteq T$. Let $Y$ be the set of all vertices $v$ in $G$ for which $v_1$ or $v_2$ is in $S\cap T$. Then $\abs{Y}\le \ell-1$. Let $Y':=\{v_1, v_2:v\in Y\}$. Clearly, $S\cap T\subseteq Y'$.

We claim that $G-Y$ has no odd $X$-walk. Assume there is an odd $X$-walk $(q_1, q_2, \ldots, q_m)$ in $G-Y$. Then $((q_1)_1, (q_2)_2, (q_3)_1, \ldots, (q_m)_2)$ is a walk in $D-Y'$ from $X_A$ to $X_B$. Thus, there is an $(X_A, X_B)$-path in $D-Y'$. It is a contradiction, as $D-Y'$ is a subgraph of $D-(S\cap T)$.
We conclude that $G-Y$ has no odd $X$-walk. 
\end{proof}

Now, we prove Lemma~\ref{lem:oddwalk2}.

\begin{proof}[Proof of Lemma~\ref{lem:oddwalk2}]
We apply Lemma~\ref{lem:oddwalk1} to $G$ and $X$ with $\ell=4k$. If $G$ contains a set of at most $4k-1$ vertices hitting all odd $X$-walks, then we are done.
Thus, we may assume that there are $4k$ odd $X$-walks such that every vertex of $G$ is used at most twice, including the number of repetitions in each walk. 
If there are $k$ odd $X$-walks such that each of them contains an odd cycle, then we get a half-integral packing of $k$ odd cycles.
So we may assume that there is a set $\cQ$ of at least $3k$ odd $X$-walks containing no odd cycles. 

We verify that every closed odd walk contains an odd cycle. Let $Q$ be a closed odd walk, and let $Q'=(q_1, q_2, \ldots, q_m)$ be a shortest closed odd walk in $Q$ with $q_1=q_m$. If there are no repeated vertices except endvertices, then $Q'$ is an odd cycle. Assume that there is a pair of repeated vertices. We choose such a pair $(q_i, q_j)$ with $\abs{j-i}$ being minimum. If the length from $q_i$ to $q_j$ is odd, then $Q'$ contains an odd cycle. Otherwise, it has even length, and by removing this part, we can find a shorter closed odd walk, a contradiction.
It implies that each walk in $\cQ$ is not closed.

Let $W\in \cQ$, and 
let $W'=(w_1, w_2, \ldots, w_t)$ be a shortest odd walk in $W$ where $W$ and $W'$ have the same endvertices. We claim that $W'$ is an odd $X$-path.
If $W'$ has no repeated vertices, then $W'$ is an odd $X$-path. Assume that there is a pair of repeated vertices. We choose such a pair $(w_i, w_j)$ with $\abs{j-i}$ being minimum. If the length from $w_i$ to $w_j$ is odd, then $W'$ contains an odd cycle, a contradiction. Otherwise, it has even length, and by removing this part, we can find a shorter odd walk with the same endvertices.
It contradicts the minimality of $W'$.
As the endvertices of $W'$ are distinct, we deduce that $W'$ is an odd $X$-path.

So, $G$ contains $3k$ odd $X$-paths such that each vertex of $G$ is used in at most two of them. By greedily choosing one $X$-path and removing two possible $X$-paths sharing an endvertex with it, we can find $k$ of them that have pairwise disjoint endvertices.
\end{proof}

\section{Well-linked sets and $r$-externally-well-linked sets}\label{sec:welllinked}

In this section, we construct a useful structure from a large $r$-externally-well-linked set.
  A \emph{bramble} in a directed graph $G$ is a set $\cB$ of strongly connected subgraphs of $G$ such that 
    for all $B_1, B_2\in \cB$, $V(B_1)\cap V(B_2)\neq \emptyset$. 
    A \emph{cover} of $\cB$ is a set $X$ of vertices in $G$ such that $V(B)\cap X\neq \emptyset$ for all $B\in \cB$. The \emph{order} of $\cB$ is the minimum size of a cover of $\cB$. 

    Note that Reed~\cite{Reed99} originally defined (directed) brambles as sets $\cB$ of strongly connected subgraphs of $G$ such that  for all $B_1, B_2\in \cB$, \begin{itemize}
        \item $V(B_1)\cap V(B_2)\neq \emptyset$ or
        \item there are an edge from $B_1$ to $B_2$ and an edge from $B_2$ to $B_1$.
    \end{itemize}
    We define the order of this bramble in the same way.  
    In Section 9.3 of~\cite{digraphbook}, the authors compared these two concepts.   
    To compare them, we say that the former is a bramble of the first type, and the latter is a bramble of the second type. They argued that for $k\ge 1$, if a directed graph has a bramble of the first type of order $k$, then it has a bramble of the second type of order at least $k$, and 
    if a directed graph has a bramble of the second type of order $4k-3$, then it has a bramble of the first type of order at least $k$. We will only consider the bramble of the first type.
    
    We use the following lemma. 
    \begin{lemma}[Lemma 4.3 of \cite{KawaK2015}]\label{lem:intersecting}
	Let $G$ be a directed graph and $\cB$ be a bramble of $G$. 
	Then there is a path $P$ intersecting every set in $\cB$.
	\end{lemma}

\begin{LEMLARGEPATH}
	Let $r$ and $p$ be positive integers with $2p(p+1)\ge r$.
	If a directed graph $G$ contains an $r$-externally-well-linked set $T$ of size at least $12p(p+1)+1$, then 
	there exist a path $P$ in $G$ and $A\subseteq V(P)$ with $\abs{A}=p$ such that 
 \begin{itemize}
     \item $A$ is well-linked, and 
     \item for every subset $Z$ of $T$ of size at least $\abs{T}/2$, there is a linkage of order $p$ from $A$ to $Z$, and there is a linkage of order $p$ from $Z$ to $A$.
 \end{itemize} 
	\end{LEMLARGEPATH}
	\begin{proof}
	  Let $T$ be an $r$-externally-well-linked set of size $m\ge 12p(p+1)+1$ in a directed graph $G$.
	As $2p(p+1)\ge r$, by Lemma~\ref{lem:wls}, $T$ is $2p(p+1)$-linked.
	We construct a bramble $\cB$ of order at least $2p(p+1)$ as follows.
	By definition of a $k$-linked set, for every set $X$ of less than $2p(p+1)$ vertices in $G$, 
	$G-X$ has a unique strong component, say $C_X$, containing more than half of the vertices of~$T$.
	We define 
	\[\cB:=\{C_X:X\subseteq V(G), \abs{X}<  2p(p+1)\}.\]
	Since any two distinct sets in $\cB$ intersect on $T$, 
	$\cB$ is a bramble.
	The order of $\cB$ is at least $2p(p+1)$, because for every set $Y$ of less than $2p(p+1)$ vertices, $Y$ does not hit $C_Y$ in $\cB$. 

	By Lemma~\ref{lem:intersecting},  
	there is a path $P$ intersecting every element of $\cB$. We now find the required set $A$ in $P$. 
    We construct sequences of disjoint subpaths $P_1, \ldots, P_{2p}$ of $P$ 
	and brambles $\cB_1, \ldots, \cB_{2p}\subseteq \cB$ such that 
 \begin{itemize}
     \item $P_1, P_2, \dots, P_{2p}$ appear in this order on $P$, and  
     \item for each $i\in [2p]$, the order of $\cB_i$ is $p+1$ and $\cB_i\subseteq  \{B\in \mathcal{B}:V(B)\cap V(P_i)\neq \emptyset\}$. 
 \end{itemize}
	
    For a subpath $Q$ of $P$, we consider some subfamily $\mathcal{B}_Q$ of $\mathcal{B}$ such that $\mathcal{B}_Q\subseteq \{B\in \mathcal{B}:V(B)\cap V(Q)\neq \emptyset\}$. Clearly, $\mathcal{B}_Q$ is a bramble. We will use the fact that if 
    \begin{itemize}
        \item $Q^*$ is another subpath of $P$ with $V(Q^*)\setminus V(Q)=\{z\}$, and
        \item $\mathcal{B}_Q\subseteq \mathcal{B}_{Q^*}\subseteq \{B\in \mathcal{B}:V(B)\cap V(Q^*)\neq \emptyset\}$,
    \end{itemize}  
    then the order of $\mathcal{B}_{Q^*}$ is at most the order of $\mathcal{B}_Q$ plus one, because all sets in $\mathcal{B}_{Q^*}\setminus \mathcal{B}_Q$ can be hit by~$z$.

 Let $P_1$ be the minimal initial subpath of $P$ such that $\cB_1=\{B\in \cB: V(B)\cap V(P_1)\neq \emptyset\}$
	is a bramble of order $p+1$.

	Now, suppose that for some $i<2p$, sequences $P_1, \ldots, P_i$ and $\cB_1, \ldots, \cB_i$ have been constructed.	Let $v$ be the last vertex of $P_i$ and $s$ be the successor of $v$ in $P$.
	Let $P_{i+1}$ be the minimal subpath of $P$ starting at $s$ such that 
	\[ \cB_{i+1}=\left\{B\in \cB : V(B)\cap \left( \bigcup_{j\in [i]} V(P_j) \right)=\emptyset \text{ and } V(B)\cap V(P_{i+1})\neq \emptyset \right\} \]
	has order $p+1$.
	As $\cB$ has order $2p(p+1)$, 
	such sequences $P_1, \ldots, P_{2p}$ 
	and $\cB_1, \ldots, \cB_{2p}\subseteq \cB$ exist.
	For each $i\in [p]$, let $a_i$ be the first vertex of $P_{2i}$, and let $A=\{a_i: i\in [p]\}$.

 	We verify that $A$ is well-linked. 
	Let $X$ and $Y$ be subsets of $A$ with $\abs{X}=\abs{Y}=q$.
 Let $X=\{a_{i_t}:t\in [q]\}$ and $Y=\{a_{j_t}:t\in [q]\}$.
 Note that $q\le p$.
	We claim that there is a linkage from $X$ to $Y$ of order $q$.

	Suppose for contradiction that there is no linkage of order $q$ from $X$ to $Y$. Then by Menger's theorem, there is a separation $(C,D)$ of order less than $q$ in $G$ such that $X\subseteq C$ and $Y\subseteq D$.
	As $\abs{C\cap D}<q\le p$, for each $j\in [2p]$, $C\cap D$ is not a hitting set of $\mathcal{B}_j$. 
    Also, $C\cap D$ does not meet one of the paths in $\{P_{2i_t}:t\in [q]\}$.
     So,  
    there exist $\ell\in [q]$ and $B_1\in \cB_{2i_\ell}$ such that 
	\[(C\cap D)\cap (V(P_{2i_{\ell}})\cup V(B_1))=\emptyset.\] 
    Similarly, since $C\cap D$ does not meet one of the sets in $\{V(P_{2i_t-1})\cup \{a_{i_t}\}:t\in [q]\}$,
    there exist $\ell'\in [q]$ and $B_2\in \cB_{2j_{\ell'}-1}$ such that 
		\[(C\cap D)\cap (V(P_{2j_{\ell'}-1})\cup \{a_{j_{\ell'}}\}\cup  V(B_2))=\emptyset.\]
    On the other hand, by the construction of $\mathcal{B}$, $B_1$ and $B_2$ intersect. Since each of $B_1$ and $B_2$ is strongly connected, $B_1\cup B_2$ is also strongly connected. This implies that there is a path from $a_{i_{\ell}}$ to $a_{j_{\ell'}}$ in 
    \[B_1\cup B_2\cup P_{2i_{\ell}}\cup G[V(P_{2j_{\ell'}-1})\cup \{a_{j_{\ell'}}\}],\] which avoids $C\cap D$, a contradiction.
    We conclude that $A$ is well-linked.

	Lastly, we verify the second bullet.
	Let $Z\subseteq T$ with $\abs{Z}\ge \abs{T}/2$.
	Suppose that there is no linkage of order $p$ from $A$ to $Z$ in $G$.
	Then, by Menger's theorem, 
	there is a separation $(C, D)$ of order less than $p$ with $A\subseteq C$ and $Z\subseteq D$.
	
	As $\abs{C\cap D}<p$, 
	there exist $\ell\in [p]$ and $B\in \cB_{2i_\ell}$ such that $(C\cap D)\cap (V(P_{2i_{\ell}})\cup V(B))=\emptyset$.
	Since $a_{i_{\ell}}\in C\setminus D$, we have $V(B)\subseteq C\setminus D$ and  
	$B$ does not intersect $Z\subseteq D$. 
	It contradicts the fact that every set of $\cB$ contains more than half of the vertices in $T$.
	
	We conclude that there is a linkage of order $p$ from $A$ to $Z$, 
	and in the same way, we can show that there is a linkage of order $p$ from $Z$ to $A$.
	\end{proof}

\section{A half-integral Erd\H{o}s-P\'osa theorem for odd cycles}\label{sec:main}

 In this section,  we prove Theorem~\ref{thm:main}.
 
   We verify that if $\alpha_{k-1}\neq\infty$ and a directed graph $G$ with $\nu_2(G)<k$ has a hitting set $T$ of directed odd cycles with $\abs{T}=\tau(G)$, then $T$ is $(2\alpha_{k-1})$-externally-well-linked.

\begin{figure}
  \centering
  \begin{center}
    \tikzstyle{v}=[circle,draw,fill=black!50,inner sep=0pt,minimum width=4pt]

    \begin{tikzpicture}
      \draw[rounded corners, thick] (2,0) -- (6,0) -- (6,4) --(0,4)--(0,0)--(2,0);
    \draw[rounded corners, thick, double] (1,0.3)--(2,0.3)--(2,3.7)--(0.3,3.7)--(0.3,0.3)--(1,0.3);
    \draw[rounded corners, very thick, dashed] (-0.4,3.8)--(-0.4,1.4)--(6.4,1.4)--(6.4,3.8);
    \draw[rounded corners, very thick, dashed] (-0.2,0.2)--(-0.2,2.6)--(6.2,2.6)--(6.2,0.2);
    \draw[thick] (0.3,2.5)--(2,2.5);
    \draw[thick] (0.3,1.5)--(2,1.5);
    \draw[rounded corners, thick] (1,2.5)--(5.7,2.5)--(5.7,1.5)--(1,1.5);
      \draw(2.3,0.8) node [label=below:$T$]{};
      \draw(1,3.5) node [label=below:$A$]{};
      \draw(1,2.4) node [label=below:$Z$]{};
      \draw(1,1.3) node [label=below:$B$]{};
      \draw(5,2.4) node [label=below:$W$]{};
      \draw(6.8,3.5) node [label=below:$X$]{};
      \draw(6.8,1.3) node [label=below:$Y$]{};

    \end{tikzpicture}
  \end{center}
  \caption{The sets $A, B, Z, W$ and the separation $(X,Y)$ defined in Lemma~\ref{lem:welllinked}.}
\label{fig:welllinked}
\end{figure}
    \begin{lemma}\label{lem:welllinked}
    Let $k\ge 2$ be an integer such that $\alpha_{k-1}\neq \infty$. Let $G$ be a directed graph with $\nu_2(G) <k$ and let $T\subseteq V(G)$ with $|T|=\tau(G)$ meeting all odd cycles in $G$. Then $T$ is $(2\alpha_{k-1})$-externally-well-linked.
    \end{lemma}
    \begin{proof}
    Let $A, B\subseteq T$ be disjoint sets with $\abs{A}=\abs{B}=r\ge 2\alpha_{k-1}$. We claim that there is a linkage in $G$ from $A$ to $B$ of order $r$ containing no vertex in $T\setminus (A\cup B)$. 
    Suppose that there is no such a linkage.  
    
    Let $Z=T\setminus (A\cup B)$.
    By Menger's theorem applied to $G-Z$, 
    there is a separation $(X, Y)$ of $G$ with $A\subseteq X$, $B\subseteq Y$ such that 
    $Z\subseteq X\cap Y$ and $\abs{ (X\cap Y)\setminus Z} <r$.
    Let $W:=(X\cap Y)\setminus Z$.
    See Figure~\ref{fig:welllinked}
    for an illustration. 
    
    Let $T_A:= (T\setminus A)\cup W$ and $T_B:=(T\setminus B)\cup W$. 
    Note that \[\abs{W}= \abs{(X\cap Y)\setminus Z}<r=\abs{A}.\]
    Therefore, $\abs{T_A}<\abs{T}=\tau(G)$ and by a similar reason, $\abs{T_B}<\abs{T}=\tau(G)$. Thus, none of $T_A$ and $T_B$ is a hitting set for odd cycles.
    
    It means that there are an odd cycle $C_A$ in $G-T_A$, 
    and an odd cycle $C_B$ in $G-T_B$.
    Since $T$ is a hitting set for odd cycles, 
    $C_A$ must contain a vertex of $A$ and $C_B$ must contain a vertex of $B$.
     So, $G-Y$ contains $C_A$ and $G-X$ contains $C_B$ while $V(G-Y)\cap V(G-X)=\emptyset$.
     
    By the definition of $\alpha_{k-1}$, 
    $G-Y$ has a hitting set $M_Y$ of size at most $\alpha_{k-1}$, 
    and $G-X$ has a hitting set $M_X$ of size at most $\alpha_{k-1}$.
    Since $A$ and $B$ are disjoint, $\abs{T}-\abs{Z}= 2r$.
    It implies that $M_X\cup M_Y\cup (X\cap Y)$ is a hitting set for odd cycles in $G$ of size at most 
    \[2\alpha_{k-1}+((r-1)+\abs{Z})= 2\alpha_{k-1}+(\abs{T}-r)-1.\]
    So, $\tau(G)\le 2\alpha_{k-1}+\tau(G)-r-1$ 
    and $r<2\alpha_{k-1}$, which contradicts the choice of $r$.
    \end{proof}
    
    As we discussed in the introduction, 
    we will consider a set $N$ of nails in a bipartite cylindrical wall $W'$ where $N$ is contained in the same part of the bipartition of $W'$, 
    and apply Lemma~\ref{lem:oddwalk2} for odd $N$-walks. 
    When Lemma~\ref{lem:oddwalk2} outputs a hitting set for odd $N$-walks, 
    the following proposition will imply that there is a small hitting set for odd cycles.

    \begin{proposition}\label{prop:smallsep}
    Let $r$, $t$, and $w$ be positive integers with $w\ge 2t$ and $t\ge r$.
    Let $G$ be a directed graph, and let $T$ be a set of at least $6t-4$ vertices in $G$ such that 
    $T$ is a hitting set of odd cycles, and  
    it is $r$-externally-well-linked.
    Let $W$ be a cylindrical wall of order $w$ in $G$ satisfying that 
    for every subset $Z$ of $T$ of size at least $\abs{T}/2$ and every set $F$ of $w$ nails in $W$, there is a linkage of order at least $w/2$ from $Z$ to $F$ in $G$, and there is a linkage of order at least $w/2$ from $F$ to $Z$ in $G$.
    Let $N$ be a set of nails of $W$ with $\abs{N}\ge 2w^2$.
    
   If $G$ has a set of less than $t$ vertices hitting all odd $N$-walks, 
   then it has a set of at most $3(t-1)$ vertices
   hitting all odd cycles.
    \end{proposition}
    \begin{proof}
    Let $X$ be a set of less than $t$ vertices in $G$ hitting all odd $N$-walks.
    Let $\{H_1, H_2, \ldots, H_m\}$ be the set of all strong components of $G-X$, and assume that it is ordered in an acyclic ordering, that is, for distinct $i, j\in [m]$, there can be an edge from $H_i$ to $H_j$ only if $i<j$.

    As $t\ge r$ and $T$ is an $r$-externally-well-linked set of size at least $6t-4$,
    by Lemma~\ref{lem:wls}, $T$ is $t$-linked. 
    Since $T$ is $t$-linked and $X$ has size less than $t$, 
    $G-X$ has a unique strong component, say $H_x$, having more than half of the vertices in $T$.
    Note that $H_x$ contains at least $t$ vertices of $T$, as $3t-2\ge t$.
   If $\bigcup_{i\in [x-1]} V(H_i)$ contains at least $t$ vertices of $T$, 
   then since $T$ is $r$-externally-well-linked and $t\ge r$, there is a linkage of order $t$ from $T\cap V(H_x)$ to $T\cap (\bigcup_{i\in [x-1]} V(H_i))$. But every path in the linkage must contain a vertex of $X$, and it contradicts the assumption that $\abs{X}<t$. Therefore,
   $\bigcup_{i\in [x-1]} V(H_i)$ contains less than $t$ vertices of $T$, and similarly, $\bigcup_{i\in [m]\setminus [x]} V(H_i)$ contains less than $t$ vertices of $T$.

    As $w\ge 2t$, there is a set $\mathcal{C}$ of at least $w-(t-1)\ge w/2+1$ columns of $W$ containing no vertex of $X$.
    We claim that for each $C\in \mathcal{C}$, $C$ is contained in $H_x$.
    Let $C\in \mathcal{C}$ and 
    let $F$ be a set of $w$ nails of $W$ that are contained in $C$. Note that $V(H_x)\cap T$ is a subset of $T$ of size at least $\abs{T}/2$.
    So, by the assumption, 
    \begin{itemize}
        \item there is a linkage of order at least $w/2\ge t$ from $V(H_x)\cap T$ to $F$ in $G$, and
        \item there is a linkage of order at least $w/2\ge t$ from $F$ to $V(H_x)\cap T$ in $G$.
    \end{itemize}
    Since $C$ does not contain a vertex of $X$ and $C$ is strongly connected, $C$ is contained in one of the strong components in $\{H_i:i\in [m]\}$. But if $C$ is contained in a strong component other than $H_x$, then either 
    \begin{itemize}
        \item there is no linkage of order $t$ from $V(H_x)\cap T$ to $F$ in $G$, or 
        \item there is no linkage of order $t$ from $F$ to $V(H_x)\cap T$ in $G$.
    \end{itemize} 
    This is a contradiction. 
    Therefore, the claim holds.
    
    Note that each column contains $4w$ nails and the columns not in $\mathcal{C}$ contain at most $4w(w/2-1)$ nails in total. Since $2w^2\ge 4w(w/2-1)+2$, $H_x$ contains at least two nails of $W$ in $N$, say $v$ and $z$.

    We claim that $H_x$ contains no odd cycle.
    Suppose for contradiction that $H_x$ contains an odd cycle $H$.  
    Since $H_x$ is strongly connected, there is a path $P_v$ from $v$ to $H$ in $H_x$, 
    and there is a path $P_z$ from $H$ to $z$ in $H_x$.
    In $H\cup P_v\cup P_z$, there are two  walks from $v$ to $z$, namely, 
    one is obtained by using the shortest path in $H$ from the endvertex of $P_v$ in $H$ to the endvertex of $P_z$ in $H$, 
    and the other one is obtained by traversing $H$  one more time.
    As $H$ is an odd cycle, the two walks have different parities.
   So $G$ contains an odd walk between two nails of $W$ that is contained in $H_x$, which contradicts the assumption that $X$ hits all odd $N$-walks. Thus, $H_x$ has no odd cycle.

    For other strong components $H_y\neq H_x$, $T\cap V(H_y)$ intersects all odd cycles in $H_y$. 
    Therefore, $(T\cap (V(G)\setminus V(H_x))) \cup X$ hits all odd cycles. We remind that  $\bigcup_{i\in [x-1]} V(H_i)$ contains less than $t$ vertices of $T$, and $\bigcup_{i\in [m]\setminus [x]} V(H_i)$ contains less than $t$ vertices of $T$. Thus, $(T\cap (V(G)\setminus V(H_x))) \cup X$ has size at most~$3(t-1)$.
    \end{proof}

     By Proposition~\ref{prop:smallsep}, we may assume that Lemma~\ref{lem:oddwalk2} outputs a large half-integral packing of odd paths whose endvertices are distinct nails of $W'$. The rest of this section devotes to find a half-integral packing of $k$ odd cycles from it.

\begin{proposition}\label{prop:manywalks}
	There is a function $g_{path}: \mathbb{N}\rightarrow \mathbb{R}$ satisfying the following.
    Let $k$ be a positive integer, and let $W$ be a bipartite cylindrical wall of order at least $(2k+3)(6g_{path}(k)+1)$ in a directed graph $G$.
    Let $N$ be a set of nails of $W$ that are contained in the same part of the bipartition of $W$.
    Let $\cU$ be a half-integral packing of $12(g_{path}(k)-1)+1$ odd $N$-paths in $G$ such that 
    the endvertices of paths in $\cU$ are disjoint.
    Then $G$ contains a half-integral packing of $k$ odd cycles.
\end{proposition}

We prove two auxiliary lemmas, and then prove Proposition~\ref{prop:manywalks}.
	Let $W$ be a bipartite cylindrical wall in a directed graph $G$.
For $v,w\in V(W)$, a walk $P$ in $G$ from $v$ to $w$ is \emph{parity-breaking} for~$W$ if the parity of the length of $P$ is different from the parity of a path from $v$ to $w$ in $W$. If the parities are the same, then we say that $P$ is \emph{parity-preserving} for $W$. 

	\begin{lemma}\label{lem:breaking}
	Let $G$ be a directed graph, and let $W$ be a bipartite cylindrical wall in $G$. If $P$ is a parity-breaking walk for $W$ from $a$ to $b$, then either $G[V(P)]$ contains an odd cycle, or it contains a parity-breaking path for $W$ from $a$ to $b$. 
	\end{lemma}
	\begin{proof}
	Let $Q=(q_1, q_2, \ldots, q_m)$ be a shortest parity-breaking walk from $a$ to $b$ contained in $G[V(P)]$.
	If $Q$ has no repeated vertices, then $Q$ is a parity-breaking path. Assume that there is a pair of repeated vertices. We choose such a pair $(q_i, q_j)$ with $\abs{j-i}$ is minimum. If the length from $q_i$ to $q_j$ is odd, then $G[V(P)]$ contains an odd cycle. Otherwise, it has even length, and by removing this part, we can find a shorter parity-breaking walk with same endvertices.
It contradicts the minimality of $Q$.
    \end{proof}

    \begin{figure}
  \centering
  \begin{center}
    \tikzstyle{v}=[circle,draw,fill=black!50,inner sep=0pt,minimum width=4pt]

    \begin{tikzpicture}[scale=1, decoration={
    markings,
    mark=at position 0.5 with {\arrow{>}}}]

    \draw[thick, postaction={decorate}] (1, 1.5)-- (3, 1.5);
    \draw[thick, postaction={decorate}] (1, 2)-- (3, 2);
    \draw[thick, postaction={decorate}] (1, 2.5)-- (3, 2.5);

    \draw[thick, postaction={decorate}] (7, 1.5)-- (9, 1.5);
    \draw[thick, postaction={decorate}] (7, 2)-- (9, 2);
    \draw[thick, postaction={decorate}] (7, 2.5)-- (9, 2.5);

    \draw[thick, postaction={decorate}] (3, 2.5) to [out=70,in=180] (4.5, 4.3);
    \draw[thick, postaction={decorate}] (4.5, 4.3) to [out=0,in=135] (6, 3);
    \draw[thick, postaction={decorate}] (6, 3) to [out=-45,in=170] (7, 2.5);

    \draw[thick, postaction={decorate}] (3, 2) to [out=40,in=-100] (4.2, 4.18);
    \draw[thick, postaction={decorate}] (4.2, 4.35) to [out=80,in=180] (5, 5.3);
    \draw[thick, postaction={decorate}] (5, 5.3) to [out=0,in=100] (6, 3.1);
    \draw[thick, postaction={decorate}] (6, 2.9) to [out=-80,in=180] (6.5, 1.5);
    \draw[thick, postaction={decorate}] (6.5, 1.5) -- (7, 1.5);

    \draw[thick, postaction={decorate}] (3, 1.5) to [out=30,in=-180] (4.5, 2.5);
    \draw[thick, postaction={decorate}] (4.5, 2.5) -- (5.95, 2.5);
    \draw[thick, postaction={decorate}] (6.15, 2.5) to [out=0,in=140] (7, 2);

    \draw[rounded corners, thick] (2,0) -- (10,0) -- (10,3.5) --(0,3.5)--(0,0)--(2,0);

    \draw[rounded corners, thick] (1,0.3)--(1.5,0.3)--(1.5,3.2)--(0.5,3.2)--(0.5,0.3)--(1,0.3);
    \draw[rounded corners, thick] (3,0.3)--(3.5,0.3)--(3.5,3.2)--(2.5,3.2)--(2.5,0.3)--(3,0.3);

    \draw[rounded corners, thick] (7,0.3)--(7.5,0.3)--(7.5,3.2)--(6.5,3.2)--(6.5,0.3)--(7,0.3);
    \draw[rounded corners, thick] (9,0.3)--(9.5,0.3)--(9.5,3.2)--(8.5,3.2)--(8.5,0.3)--(9,0.3);

      \draw(1,1) node [label=below:$A$]{};
      \draw(3,1) node [label=below:$B$]{};
      \draw(7,1) node [label=below:$C$]{};
      \draw(9,1) node [label=below:$D$]{};
      \draw(10.5,0.8) node [label=below:$W$]{};

    \draw(2,1.5) node [label=below:$\mathcal{Q}$]{};
    \draw(8,1.5) node [label=below:$\mathcal{R}$]{};
    \draw(5,2) node [label=below:$\mathcal{U}$]{};
     
    \draw (1, 1.5) node [v] () {};
    \draw (1, 2) node [v] () {};
    \draw (1, 2.5) node [v] () {};
    \draw (3, 1.5) node [v] () {};
    \draw (3, 2) node [v] () {};
    \draw (3, 2.5) node [v] () {};
    \draw (7, 1.5) node [v] () {};
    \draw (7, 2) node [v] () {};
    \draw (7, 2.5) node [v] () {};
    \draw (9, 1.5) node [v] () {};
    \draw (9, 2) node [v] () {};
    \draw (9, 2.5) node [v] () {};
    
    \end{tikzpicture}
  \end{center}
  \caption{The sets $A, B, C, D$ in a wall $W$ and the linkages $\mathcal{Q}, \mathcal{R}$ and the half-integral linkage $\mathcal{U}$ defined in Lemma~\ref{lem:expandinghalf}.}
\label{fig:expandinghalf}
\end{figure}

    \begin{figure}
  \centering
  \begin{center}
    \tikzstyle{v}=[circle,draw,fill=black!50,inner sep=0pt,minimum width=4pt]
  \tikzstyle{w}=[rectangle, draw, solid, fill=black, inner sep=0pt, minimum width=4pt, minimum height=4pt]
  
    \begin{tikzpicture}[scale=1, decoration={
    markings,
    mark=at position 0.5 with {\arrow{>}}}]

    \foreach \x in {1,2,3,4,5}{
    \draw[thick, postaction={decorate}] (\x, 3)-- (\x+1, 3);
    \draw[thick, postaction={decorate}] (\x+1, 2)-- (\x, 2);
    \draw[thick, postaction={decorate}] (\x, 1)-- (\x+1, 1);
    \draw[thick, postaction={decorate}] (\x+1, 0)-- (\x, 0);
    }
    \foreach \x in {2,4}{
    \draw[thick, postaction={decorate}] (\x, 3.8)-- (\x, 3);
    \draw[thick, postaction={decorate}] (\x+1, 3)-- (\x+1, 2);
    \draw[thick, postaction={decorate}] (\x, 2)-- (\x, 1);
    \draw[thick, postaction={decorate}] (\x+1, 1)-- (\x+1, 0);
    \draw[thick, postaction={decorate}] (\x, 0)-- (\x, -0.8);
    }

    \foreach \x in {2,3,4,5}{
    \foreach \y in {0,1,2,3}{
    \draw (\x, \y) node [v] () {};
    }
    }

     \draw[thick, postaction={decorate}] (2, 3) to[out=60, in=180] (2.7, 4.5);
     \draw[thick, postaction={decorate}] (2.7, 4.5) to[out=0, in=180] (3.3, 4.5);
     \draw[thick, postaction={decorate}] (3.3, 4.5) to[out=0, in=120] (4, 3);
    \draw (2.6, 4.5) node [v] () {};
    \draw (3.4, 4.5) node [v] () {};
    
     \draw[thick, postaction={decorate}] (5, 2) to[out=60, in=0] (4, 5);
     \draw[thick, postaction={decorate}] (4, 5) to[out=180, in=0] (3, 5);
     \draw[thick, postaction={decorate}] (3, 5) to[out=180, in=0] (2, 5);
     \draw[thick, postaction={decorate}] (2, 5) to[out=180, in=90] (0.5, 3.5);
     \draw[thick, postaction={decorate}] (0.5, 3.5) to[out=-90, in=140] (2,2);
    \draw (4, 5) node [v] () {};
    \draw (3, 5) node [v] () {};
    \draw (2, 5) node [v] () {};
    \draw (0.5, 3.5) node [v] () {};

    \draw(2,1) node [label=below:$q\in A$]{};
    \draw(5,0) node [label=below:$r\in D$]{};
    \draw[color=black, very thick](2,1) circle (0.2);
    \draw[color=black, very thick](5,0) circle (0.2);
        
    \draw(3.5,-1) node [label=below:$G$]{};
    \end{tikzpicture} \qquad\qquad
       \begin{tikzpicture}[scale=1, decoration={
    markings,
    mark=at position 0.5 with {\arrow{>}}}]

     \foreach \x in {1,5}{
    \draw[thick, postaction={decorate}] (\x, 3)-- (\x+1, 3);
    \draw[thick, postaction={decorate}] (\x+1, 2)-- (\x, 2);
    \draw[thick, postaction={decorate}] (\x, 1)-- (\x+1, 1);
    \draw[thick, postaction={decorate}] (\x+1, 0)-- (\x, 0);
    }

    \foreach \x in {2,3,4}{
    \draw[thick, postaction={decorate}] (\x, 3)-- (\x+0.7, 3);
    \draw[thick, postaction={decorate}] (\x+0.5, 3)-- (\x+1, 3);
    \draw[thick, postaction={decorate}] (\x+1, 2)-- (\x+0.3, 2);
    \draw[thick, postaction={decorate}] (\x+0.5, 2)-- (\x, 2);
    \draw[thick, postaction={decorate}] (\x, 1)-- (\x+0.7, 1);
    \draw[thick, postaction={decorate}] (\x+0.5, 1)-- (\x+1, 1);
    \draw[thick, postaction={decorate}] (\x+1, 0)-- (\x+0.3, 0);
    \draw[thick, postaction={decorate}] (\x+0.5, 0)-- (\x, 0);
    \draw (\x+0.5, 3) node [w] () {};
    \draw (\x+0.5, 2) node [w] () {};
    \draw (\x+0.5, 1) node [w] () {};
    \draw (\x+0.5, 0) node [w] () {};
    }
    \foreach \x in {2,4}{
    \draw[thick, postaction={decorate}] (\x, 3.8)-- (\x, 3);
    \draw[thick, postaction={decorate}] (\x+1, 3)-- (\x+1, 2.3);
    \draw[thick, postaction={decorate}] (\x+1, 2.5)-- (\x+1, 2);
    \draw[thick, postaction={decorate}] (\x, 2)-- (\x, 1.3);
    \draw[thick, postaction={decorate}] (\x, 1.5)-- (\x, 1);
    \draw[thick, postaction={decorate}] (\x+1, 1)-- (\x+1, 0.3);
    \draw[thick, postaction={decorate}] (\x+1, 0.5)-- (\x+1, 0);
    \draw[thick, postaction={decorate}] (\x, 0)-- (\x, -0.8);
    \draw (\x+1, 2.5) node [w] () {};
    \draw (\x, 1.5) node [w] () {};
    \draw (\x+1, 0.5) node [w] () {};
    }

    \foreach \x in {2,3,4,5}{
    \foreach \y in {0,1,2,3}{
    \draw (\x, \y) node [v] () {};
    }
    }

    \draw (1.2, 1.7) node [w] (v1) {};
    \draw (0.4, 1.7) node [w] (v2) {};
    \draw (2, 1) node [v] (v) {};
    \draw[thick, postaction={decorate}] (v2)-- (v1);
    \draw[thick, postaction={decorate}] (v1)-- (v);
    \draw(1.2,1.8) node [label=below:$q^2$]{};
    \draw(0.4,1.8) node [label=below:$q^1$]{};

    \draw (5.8, 0.7) node [w] (v3) {};
    \draw (6.6, 0.7) node [w] (v4) {};
    \draw (5, 0) node [v] (v5) {};
    \draw[thick, postaction={decorate}] (v5)-- (v3);
    \draw[thick, postaction={decorate}] (v3)-- (v4);
    \draw(5.8,0.8) node [label=below:$r^2$]{};
    \draw(6.6,0.8) node [label=below:$r^1$]{};

    \draw(2,1) node [label=below:$q\in A$]{};
    \draw(5,0) node [label=below:$r\in D$]{};
    \draw[color=black, very thick](2,1) circle (0.2);
    \draw[color=black, very thick](5,0) circle (0.2);

    \draw[thick, postaction={decorate}](2,3) to[bend left=45] (4,3);
    \draw[thick, postaction={decorate}](5,2) to[out=140,in=0] (3.5,2.7);
    \draw[thick, postaction={decorate}](3.5,2.7) to[out=180,in=40] (2,2);
    \draw (3.5, 2.7) node [w] () {};
      
    \draw(3.5,-1) node [label=below:$F_1$]{};
    \end{tikzpicture}
  \end{center}
  \caption{An illustration of the construction of $F_1$ from $G$ in Lemma~\ref{lem:expandinghalf}.}
\label{fig:construction}
\end{figure}

For the following lemma, see Figure~\ref{fig:expandinghalf} for an illustration of the initial setting.
	\begin{lemma}\label{lem:expandinghalf}
		Let $k$ and $m$ be positive integers. Let $G$ be a directed graph, $W$ be a bipartite cylindrical wall in $G$, and let $A, B, C, D$ be disjoint subsets of $V(W)$ of size $m$.
		Let $\cQ$ be a linkage of order $m$ from $A$ to $B$ in $W$, 
		and let $\cR$ be a linkage of order $m$ from $C$ to  $D$ in $W$.
		Let $\cU$ be a half-integral packing of $m$ parity-breaking paths from $B$ to $C$ in $G$ such that the first vertices of paths in $\cU$ are all distinct and the last vertices of paths in $\cU$ are all distinct.
		If $m\ge 8k$, then there is either 
		\begin{itemize}
		    \item a half-integral packing of $k$ odd cycles, or
		    \item a half-integral packing of $k$ parity-breaking paths from $A$ to $D$ in $(\bigcup\cQ)\cup (\bigcup\cR)\cup (\bigcup\cU)$ such that the first vertices of the paths are all distinct and the last vertices of the paths are all distinct.
		\end{itemize} 		
	\end{lemma}
	\begin{proof}
	We construct a directed graph $F_1$ starting from the vertex set $V(W)$  and the empty edge set
	as follows. See Figure~\ref{fig:construction} for an illustration.
	\begin{itemize}
		\item For every edge $(u, v)$ in $E(W)$, we add a new vertex $x_{uv}$ and two edges $(u, x_{uv})$ and $(x_{uv}, v)$. 
		\item For every $W$-path $P$ from a vertex $u$ to a vertex $v$ that is a subpath of some path in $\cU$, 
		if $P$ is parity-breaking, then we add an edge $(u,v)$, and 
		otherwise, we add a vertex $z_{uv}$ and two edges $(u,z_{uv})$ and $(z_{uv}, v)$.
        We denote it by $\widehat{P}$.
		\item For every $q\in A$, we add two new vertices $q^1$ and $q^2$ and add edges $(q^1, q^2)$ and $(q^2, q)$.
		\item For every $r\in D$, we add two new vertices $r^1$ and $r^2$ and add edges $(r, r^2)$ and $( r^2, r^1)$.
	\end{itemize}
	We assign $A_1:=\{q^1: q\in A\}$ and $D_1:=\{r^1: r\in D\}$. 

    We say that a walk $J$ between two vertices of $W$ is \emph{pure} if every $W$-subpath of $J$ is a subpath of some path in $\mathcal{U}$.
    For a pure walk $J$ from a vertex $q\in A$ to a vertex $r\in D$ in $(\bigcup\cQ)\cup (\bigcup\cR)\cup (\bigcup\cU)$, we obtain a walk $\widehat{J}$ in $F_1$ by, 
    \begin{itemize}
        \item for each edge $(u,v)$ of $W$ in $J$, replacing it with $(u,x_{uv}), (x_{uv},v)$,
        \item for each $W$-path $P$ contained in $J$, replacing it with $\widehat{P}$, and
        \item adding vertices $q^1, q^2, r^1, r^2$ and edges $(q^1,q^2), (q^2,q), (r,r^2), (r^2,r^1)$.
    \end{itemize} 

    Let $J$ be a pure walk from a vertex $q\in A$ to a vertex $r\in D$ in $(\bigcup\cQ)\cup (\bigcup\cR)\cup (\bigcup\cU)$. 
    Then $J$ is parity-breaking if and only if the number of parity-breaking $W$-paths contained in $J$ is odd.
    As we replace each edge of $W$ and each parity-preserving $W$-path by a path of length $2$ while we replace each parity-breaking $W$-path by a path of length $1$, the parity of the number of parity-breaking paths contained in $J$ is the parity of $\widehat{J}$ in $F_1$. Thus, $J$ is parity-breaking if and only if $\widehat{J}$ is odd.

    Note that combining a path $Q\in \mathcal{Q}$ from $a$ to $b$, and a path $U\in \mathcal{U}$ from $b$ to $c$, and a path $R\in \mathcal{R}$ from $c$ to $d$, we get a pure walk from $a$ to $d$. As two paths in $\cU$ may share a vertex on a path in $\cQ\cup \cR$, there is a set of $m$ pure walks from $A$ to $D$ such that 
    \begin{itemize}
        \item every vertex of $G$ is used at most $4$ times,
        \item the first vertices of them are all distinct, and the last vertices of them are all distinct.
    \end{itemize} 
	Thus, there is a set of $m$ odd walks from $A_1$ to $D_1$ in $F_1$ such that 
	\begin{itemize}
	    \item every vertex of $F_1$ is used at most $4$ times, and 
	    \item for each vertex $w\in A_1\cup D_1$, there is at most one walk containing $w$ in the $m$ odd walks.
	\end{itemize}

    Now, we obtain a bipartite directed graph $F_2$ with bipartition $(X, Y)$ from $F_1$ such that 
\begin{itemize}
    \item $X=\{v_1: v\in V(F_1)\}$ and $Y=\{v_2: v\in V(F_1)\}$, 
    \item $E(F_2)=\{(v_1,w_2), (v_2,w_1):(v,w)\in E(F_1)\}$.
\end{itemize}
Let $A_2:=\{q^1_1:q\in A\}$ and $D_2:=\{r^1_2:r\in D\}$.

    Observe that a walk from $A_2$ to $D_2$ in $F_2$ corresponds to an odd walk from $A_1$ to $D_1$ in $F_1$.
	Thus, 
	there is a set of $m$ walks from $A_2$ to $D_2$ in $F_2$ such that every vertex is used at most $4$ times.
	So, there is a $1/4$-integral packing of $m$ paths from $A_2$ to $D_2$ in $F_2$.
	Since $m\ge 8k$, by Lemma~\ref{lem:mintegral}, 
	there is a linkage of order $2k$ from $A_2$ to $D_2$ in $F_2$.
	
	It implies that there is a set $\mathcal{L}_1$ of $2k$ odd walks from $A_1$ to $D_1$ in $F_1$ such that 
	\begin{itemize}
	    \item every vertex of $F_1$ is used at most twice, 
	    \item the first vertices of paths in $\mathcal{L}_1$ are all distinct, and  
	    \item the last vertices of paths in $\mathcal{L}_1$ are all distinct.
	\end{itemize}
	Furthermore, 
	there is a set $\mathcal{L}_2$ of $2k$ parity-breaking walks from $A$ to $D$  in $(\bigcup\cQ)\cup (\bigcup\cR)\cup (\bigcup\cU)$ such that 
	\begin{itemize}
	    \item every vertex of $G$ is used at most twice, 
	     \item the first vertices of paths in $\mathcal{L}_2$ are all distinct, and  
	    \item the last vertices of paths in $\mathcal{L}_2$ are all distinct.
	\end{itemize} 
	Now, by Lemma~\ref{lem:breaking}, 
	either there is a half-integral packing of $k$ odd cycles, or there is a half-integral packing of $k$ parity-breaking paths from $A$ to $D$ in $(\bigcup\cQ)\cup (\bigcup\cR)\cup (\bigcup\cU)$, where the first vertices are all distinct, and the last vertices are all distinct.
	\end{proof}

\begin{proof}[Proof of Proposition~\ref{prop:manywalks}]
We set 
\begin{itemize}
    \item $g_3(k)=8k$
	\item $g_2(k)=4g_3(k)$,
	\item $g_1(k)=(2g_2(k)-1)^2+1$,
	\item $g_{path}(k)=g(k)=(2g_1(k)-1)^2+1$.
\end{itemize}
Let $w$ be the order of $W$. Let $C_1, \ldots, C_w$ be the columns of $W$, and $P_1, \ldots, P_{2w}$ be the rows of $W$. We will consider indices of rows by the congruence modulo $2w$. Recall that $N^W$ is the set of all nails of $W$.

 Since every path in $\cU$ is an odd path between two nails in the same part of the bipartition of $W$, every path in $\cU$ is parity-breaking. We start with finding subpaths of some paths in $\cU$ so that they are still parity-breaking and do not intersect many $N^W$-paths in $W$. 

    \begin{claim}\label{claim:paritybreaking}
    For every $t\in [g(k)]$, there is a half-integral packing of parity-breaking paths $U_1, U_2, \ldots, U_t$ for $W$ such that for each $i\in [t]$, 
    \begin{enumerate}[(i)]
        \item\label{breaking1} $\bigcup_{j\in [i]}U_j$ intersects at most $6i$ $N^W$-paths in $W$, and
        \item\label{breaking2} 
        $\bigcup_{j\in [i-1]}U_j$ does not intersect any $N^W$-path containing an endvertex of $U_i$.
    \end{enumerate}
    \end{claim}
    \begin{pfofclaim}
    We prove the statement by induction on $1\le t\le g(k)$.
    Assume that such a set of paths $U_1, \ldots, U_{t-1}$ has been constructed for some $t\le g(k)$.  
    By Property~(\ref{breaking1}), 
    $\bigcup_{j\in [t-1]}U_j$ intersects at most $6(t-1)\le 6(g(k)-1)$ $N^W$-paths in $W$. 
    Let $\mathcal{A}$ be the set of all $N^W$-paths in $W$ that contain a vertex of $\bigcup_{j\in [t-1]}U_j$, and let $B:=\bigcup_{Q\in \cA} V(Q)$. 
    Note that $B$ contains at most $12(g(k)-1)$ nails.
    Since $\abs{\cU}=12(g(k)-1)+1$ and the endvertices of paths in $\cU$ are disjoint, 
    there is a path $U\in \cU$ such that
    the endvertices of $U$ are not contained in $B$.
    
    Let $U=u_1u_2 \cdots u_m$.
    Let $\mathcal{Q}$ be the set of all subpaths $U^*$ of $U$ of length at least $1$ where its endvertices are in $V(W)\setminus B$ and all internal vertices are not in $V(W)\setminus B$. 
    
    Note that the paths in $\mathcal{Q}$ are pairwise edge-disjoint, and $\bigcup_{Q\in \mathcal{Q}}E(Q)=E(U)$.
    Since $U$ is parity-breaking, $\mathcal{Q}$ contains at least one parity-breaking path. 
    Let $U'$ be a parity-breaking path in $\mathcal{Q}$.
    Note that every vertex of $W$ is contained in at most three $N^W$-paths.
    Since all the internal vertices of $U'$ are not contained in $V(W)\setminus B$, 
    $U_1\cup U_2\cup \cdots \cup U_{t-1}\cup U'$ intersects at most $6(t-1)+6\le 6t$ $N^W$-paths in $W$, and 
    the $N^W$-paths containing the endvertices of $U'$ are not used by paths in $U_1, \ldots, U_{t-1}$.
    Thus, the claim holds.
    \end{pfofclaim}

    By Claim~\ref{claim:paritybreaking}, there is a half-integral packing of parity-breaking paths $U_1, U_2, \ldots, U_{g(k)}$ that intersect at most $6g(k)$ $N^W$-paths.         
    Let $\mathcal{A}$ be the set of all $N^W$-paths in $W$ that contain a vertex of $\bigcup_{j\in [g(k)]}U_j$, and let $B:=\bigcup_{Q\in \cA} V(Q)$.

  	Recall that the order of $W$ is  at least $(2k+3)(6g(k)+1)$, and 
  	each $N^W$-path may intersect at most two columns and at most two rows.
  	As \[(2k+3)(6g(k)+1)-12g(k)\ge (2k+1)(6g(k)+1)+1,\]
	there is a set of $2k+2$ consecutive columns, say $C_{z+1}, C_{z+2}, \ldots, C_{z+2k+2}$, containing no vertices of $B$.
	Also, since
	\[(4k+6)(6g(k)+1)-12g(k)\ge (4k+4)(6g(k)+1)+1,\]
	there is a set of $4k+5$ consecutive rows containing no vertices of $B$.
	Among these $4k+5$ rows, we choose $4k+4$ consecutive rows $P_{y+1}, P_{y+2}, \ldots, P_{y+4k+4}$ such that 
	$P_{y+1}$ is a row traversing from $C_1$ to $C_w$.
	We define \[W^*=P_{y+1}\cup P_{y+2}\cup \cdots \cup P_{y+4k+4} \cup C_{z+1} \cup C_{z+2} \cup \cdots \cup C_{z+2k+2}.\] 
	See Figure~\ref{fig:reroute} for an illustration of $W^*$.
	Observe that $V(W^*)\cap B=\emptyset$.
   Let $\clos (W^*)$ be the subgraph of $W$ that is the union of $W^*$ and all $N^W$-paths whose both endvertices are in $W^*$.

\begin{figure}
  \centering
  \includegraphics[scale=0.7]{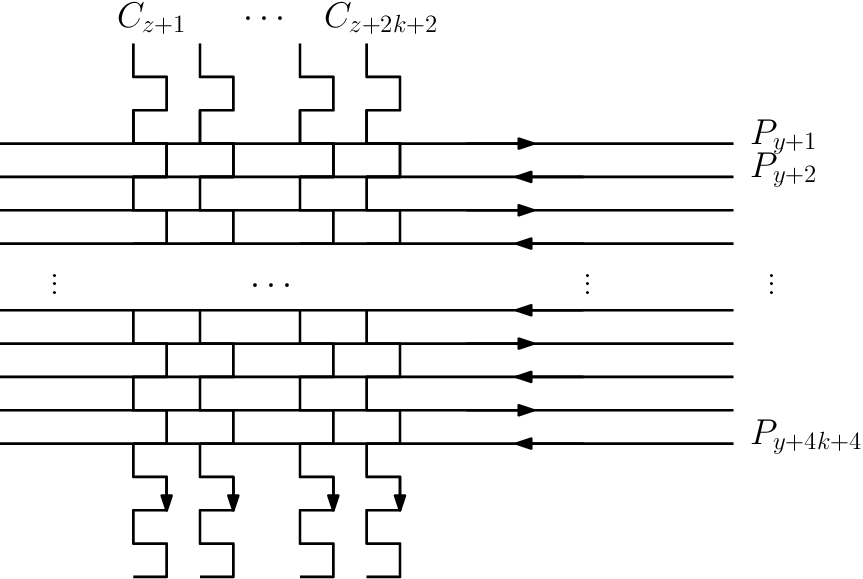}
\caption{Selected consecutive columns and rows in the proof of Proposition~\ref{prop:manywalks}.  }\label{fig:reroute}
\end{figure}

	Let $L$ be the bijection from the set of all nails of $W$ to $[w]\times [2w]\times [2]$ satisfying the following.
	\begin{itemize}
	    \item Let $i\in [w]$ and $j\in [2w]$. 
	When we traverse $C_i$ from $P_1$ to $P_{2w}$,
	$C_i$ contains two nails of each $P_j$, 
	and for the first vertex $v$, $L(v)=(i, j, 1)$ and 
	for the second vertex $v$, $L(v)=(i, j, 2)$.
	\end{itemize}

    For each $i\in [g(k)]$, we define the following.
    \begin{itemize}
        \item Let $p_i$ and $q_i$ be the endvertices of $U_i$ such that $U_i$ is a path from $p_i$ to $q_i$.
    	\item If $p_i$ is a nail, then let $p_i^*:=p_i$ and $A_i:=G[\{ p_i \}]$. 
	 Otherwise, let $p_i^*$ be the first vertex of the $N^W$-path in $W$ containing $p_i$
	               and let $A_i$ be the subpath from $p_i^*$ to $p_i$ in the $N^W$-path. Let $(a_i, b_i, c_i):=L(p_i^*)$.
	 	\item If $q_i$ is a nail, then let $q_i^*:=q_i$ and $D_i:=G[\{q_i\}]$. 
            Otherwise, let $q_i^*$ be the last vertex of the $N^W$-path in $W$ containing $q_i$ 
             and let $D_i$ be the subpath from $q_i$ to $q_i^*$ in the $N^W$-path. Let $(d_i, e_i, f_i):=L(q_i^*)$. 
    	\end{itemize}
    
   Since $g(k)=(2g_1(k)-1)^2+1$, there is a subset $I_1\subseteq [g(k)]$ of size $2g_1(k)$ such that 
    either 
   	    \begin{itemize}
   	        \item all integers in $(a_i : i\in I_1)$ are distinct, or
   	        \item all integers in $(a_i : i\in I_1)$ are the same.
	\end{itemize}
	There is a subset $I_2\subseteq I_1$ with $\abs{I_2}\ge g_1(k)$
	 such that 
	all integers in $(c_i: i\in I_2)$ are the same. Since all integers in $(c_i: i\in I_2)$ are the same, 
	all integers in $(b_i : i\in I_2)$ are distinct when all integers in $(a_i : i\in I_1)$ are the same.
	Furthermore, as $g_1(k)=(2g_2(k)-1)^2+1$, 
	there is a subset $I_3\subseteq I_2$ of size $2g_2(k)$ such that 
   	either 
	\begin{itemize}
	    \item all integers in $(d_i : i\in I_3)$ are distinct, or
	    \item all integers in $(d_i : i\in I_3)$ are the same.
	\end{itemize}  
	There is a subset $I_4\subseteq I_3$ of size $g_2(k)$ such that all integers in $(f_i : i\in I_4)$ are the same.
	Since  all integers in $(f_i : i\in I_4)$ are the same,
		all integers in $(e_i : i\in I_3)$ are distinct when all integers in $(d_i : i\in I_3)$ are the same. 

	Lastly, we take a subset $I_5\subseteq I_4$ of size $g_2(k)/4=g_3(k)$ such that 
   \begin{itemize}
   	\item if all integers in $(a_i : i\in I_4)$ are the same, then $y+4k+5\notin (b_i:i\in I_5)$ (as modulo $2w$) and $\abs{b_{i_1}-b_{i_2}}\ge 2 \pmod {2w}$ for all distinct $i_1, i_2\in I_5$, and
	\item if all integers in $(d_i : i\in I_4)$ are the same, then $y\notin (e_i:i\in I_5)$ (as modulo  $2w$) and $\abs{e_{i_1}-e_{i_2}}\ge 2 \pmod {2w}$ for all distinct $i_1, i_2\in I_5$.
   \end{itemize}
   We can greedily choose elements of $I_5$ from $I_4$.
    We choose this subset $I_5$ in a way to give enough space for paths to be picked disjointly later.

    Now, we construct a linkage $\{X_i : i\in I_5\}$ from $V(W^*)$ to $\{p_i : i\in I_5\}$ in $W$,  and a linkage $\{Y_i : i\in I_5\}$ from $\{q_i : i\in I_5\}$ to $V(W^*)$ in $W$. 
   We will apply Lemma~\ref{lem:expandinghalf}, together with the half-integral linkage $\{U_1, \ldots, U_{g(k)}\}$.

\begin{figure}
     \centering
     \begin{subfigure}[b]{0.48\textwidth}
         \centering
         \includegraphics[width=\textwidth]{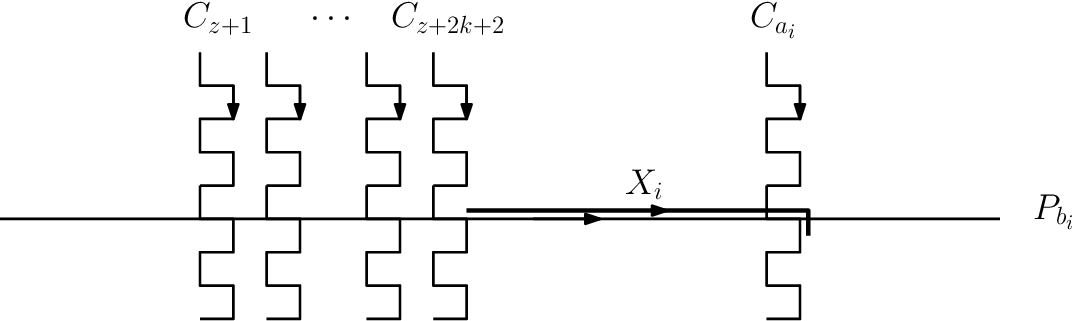}
         \caption{$b_i$ is odd and $a_i>z+2k+2$}
         \label{fig:path2}
     \end{subfigure}
     \hfill
     \begin{subfigure}[b]{0.48\textwidth}
         \centering
         \includegraphics[width=\textwidth]{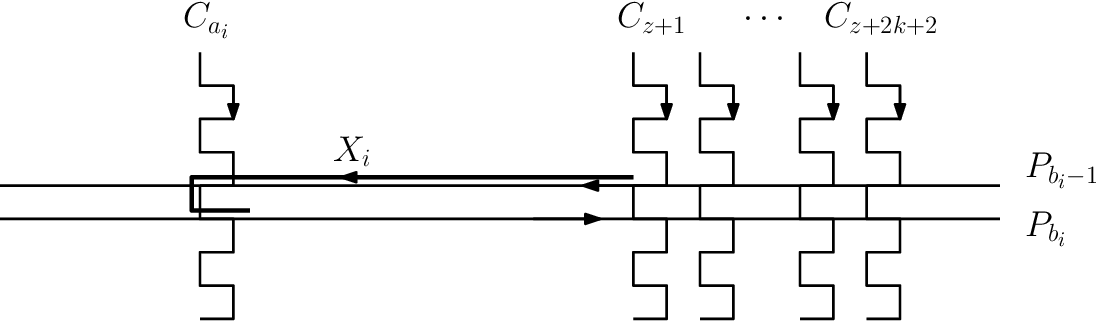}
         \caption{$b_i$ is odd and $a_i<z+1$}
         \label{fig:path3}
     \end{subfigure}
     \hfill
     \vskip 0.3cm
     \begin{subfigure}[b]{0.48\textwidth}
         \centering
         \includegraphics[width=\textwidth]{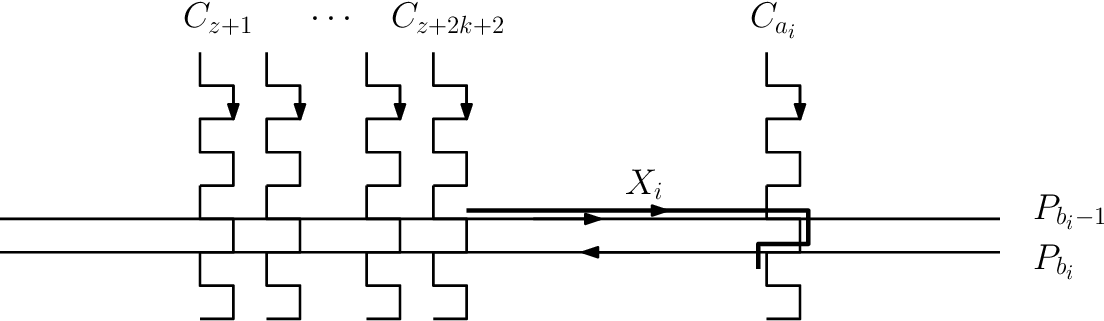}
         \caption{$b_i$ is even and $a_i>z+2k+2$}
         \label{fig:path5}
     \end{subfigure}
     \hfill
     \begin{subfigure}[b]{0.48\textwidth}
         \centering
         \includegraphics[width=\textwidth]{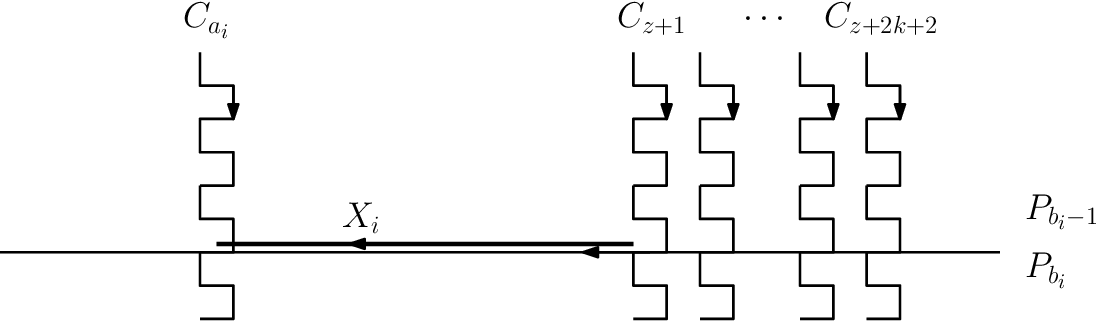}
         \caption{$b_i$ is even and $a_i<z+1$}
         \label{fig:path4}
     \end{subfigure}
        \caption{The construction of $X_i$ when all integers in $(a_i : i\in I_5)$ are the same.}
        \label{fig:lipath}
\end{figure}

   \begin{itemize}
       \item Assume that all integers in $(a_i : i\in I_5)$ are distinct.
     Let $X_i$ be the path starting at $L^{-1}(a_i, y+4k+4, 2)$, traversing to $p_i^*$ in $C_{a_i}$, and traversing to $p_i$ in $A_i$.
    \item Otherwise, all integers in $(a_i : i\in I_5)$ are the same and
	all integers in $(b_i : i\in I_5)$ are distinct. We divide into four cases. See Figure~\ref{fig:lipath} for illustrations.
	\begin{itemize}
	    \item ($b_i$ is odd and  $a_i>z+2k+2$.) Let $X_i$ be the path starting at $L^{-1}(z+2k+2, b_i, 2)$, traversing to $p_i^*$ in $P_{b_i}$, and traversing to $p_i$ in $A_i$.
	    \item ($b_i$ is odd and  $a_i<z+1$.) Let $X_i$ be the path starting at $L^{-1}(z+1, b_i-1, 2)$, traversing to $L^{-1}(a_i, b_i-1, 2)$ in $P_{b_i-1}$, traversing to $p_i^*$ in $C_{a_i}$, and traversing to $p_i$ in $A_i$.
	    \item ($b_i$ is even and  $a_i>z+2k+2$.) Let $X_i$ be the path starting at $L^{-1}(z+2k+2, b_i-1, 2)$, traversing to  $L^{-1}(a_i, b_i-1, 2)$ in $P_{b_i-1}$, traversing to $p_i^*$ in $C_{a_i}$, and traversing to $p_i$ in $A_i$.
        \item ($b_i$ is even and  $a_i<z+1$.) Let $X_i$ be the path starting at $L^{-1}(z+1, b_i, 2)$, traversing to $p_i^*$ in $P_{b_i}$, and   traversing to $p_i$ in $A_i$.
	\end{itemize} 
   \end{itemize}
    We observe that all paths in $\{X_i : i\in I_5\}$ are pairwise vertex-disjoint.
    When  all integers in $(a_i : i\in I_5)$ are distinct, 
    each path $A_i$ is starting from a vertex of $C_{a_i}$, but does not meet other column of $W$. 
    So, all paths in $\{A_i : i\in I_5\}$ are pairwise vertex-disjoint and all paths in $\{X_i : i\in I_5\}$ are pairwise vertex-disjoint.
    The case when all integers in $(a_i : i\in I_5)$ are the same is similar, 
    and  for the second and third subcases of the second case, we additionally use the fact that $y+4k+5\notin (b_i:i\in I_5)$ (as modulo $2w$) and $\abs{b_{i_1}-b_{i_2}}\ge 2 \pmod {2w}$ for all distinct $i_1, i_2\in I_5$.

   We define paths $Y_i$ in a symmetric way.
   \begin{itemize}
       \item Assume that all integers in $(d_i : i\in I_5)$ are distinct. 
       Let $Y_i$ be the path starting at $q_i$, traversing to $q_i^*$ in $D_i$, and traversing to $L^{-1}(d_i, y+1, 1)$ in $C_{d_i}$.
        \item Otherwise, all integers in $(d_i : i\in I_5)$ are the same and
	all integers in $(e_i : i\in I_5)$ are distinct. We divide into four cases.
	\begin{itemize}
	    \item ($e_i$ is odd and  $d_i>z+2k+2$.) Let $Y_i$ be the path starting at $q_i$, traversing to $q_i^*$ in $D_i$, traversing to $L^{-1}(d_i, e_i+1, 1)$ in $C_{d_i}$, and traversing to $L^{-1}(z+2k+2, e_i+1, 1)$ in $P_{e_i+1}$.
	    \item ($e_i$ is odd and  $d_i<z+1$.) Let $Y_i$ be the path starting at $q_i$, traversing to $q_i^*$ in $D_i$, and traversing to $L^{-1}(z+1, e_i, 1)$ in $P_{e_i}$.
    	   \item ($e_i$ is even and  $d_i>z+2k+2$.) Let $Y_i$ be the path starting at $q_i$, traversing to $q_i^*$ in $D_i$, and traversing to $L^{-1}(z+2k+2, e_i, 1)$ in $P_{e_i}$.
	    \item ($e_i$ is even and  $d_i<z+1$.) Let $Y_i$ be the path  starting at $q_i$, traversing to $q_i^*$ in $D_i$, traversing to $L^{-1}(d_i, e_i+1, 1)$ in $C_{d_i}$, and traversing to $L^{-1}(z+1, e_i+1, 1)$ in $P_{e_i+1}$.
	\end{itemize}
    \end{itemize}
     We observe that all paths in $\{Y_i : i\in I_5\}$ are pairwise vertex-disjoint.
    When  all integers in $(d_i : i\in I_5)$ are distinct, 
    each path $D_i$ is ending at a vertex of $C_{d_i}$, but does not meet other column of $W$. 
    So, all paths in $\{D_i : i\in I_5\}$ are pairwise vertex-disjoint and all paths in $\{Y_i : i\in I_5\}$ are pairwise vertex-disjoint.
    The case when all integers in $(d_i : i\in I_5)$ are the same is similar, 
    and  for the first and fourth subcases of the second case, we additionally use the fact that $y\notin (e_i:i\in I_5)$ (as modulo  $2w$) and $\abs{e_{i_1}-e_{i_2}}\ge 2 \pmod {2w}$ for all distinct $i_1, i_2\in I_5$.

   We apply Lemma~\ref{lem:expandinghalf} 
   for linkages $\{X_i : i\in I_5\}$, $\{Y_i : i\in I_5\}$, and a half-integral packing of parity-breaking paths $\{U_i : i\in I_5\}$.
   Since $\abs{I_5}=g_3(k)=8k$, by Lemma~\ref{lem:expandinghalf}, 
   there is either a half-integral packing of $k$ odd cycles, or
   a half-integral packing of parity-breaking paths $\mathcal{Z}=\{Z_i:i\in [k]\}$ such that 
   \begin{itemize}
    \item the first vertices of paths in $\mathcal{Z}$ are first vertices of paths in $\{X_i:i\in I_5\}$ and they are all distinct, and 
       \item the last vertices of paths in $\mathcal{Z}$ are last vertices of paths in $\{Y_i:i\in I_5\}$ and they are all distinct.
   \end{itemize}
   For each $i\in [k]$, let $s_i$ and $r_i$ be the first and last vertices of $Z_i$, respectively. Because paths in 
   $\{X_i, Y_i, U_i : i\in I_5\}$ do not use any edge of $W^*$, $\{s_i:i\in [k]\}$ and $\{r_i:i\in [k]\}$ cannot share a vertex.
   
   We construct a path $Z_i^*$ for each $i\in [k]$ in $\clos(W^*)$ so that $Z_i\cup Z_i^*$ is an odd cycle
   and $Z_i^*$ does not intersect $\bigcup_{j\in [k]} Z_j$ except the vertices in $\{r_i, s_i\}$.

\begin{figure}[t]
     \centering
     \begin{subfigure}[b]{0.48\textwidth}
         \centering
         \includegraphics[width=\textwidth]{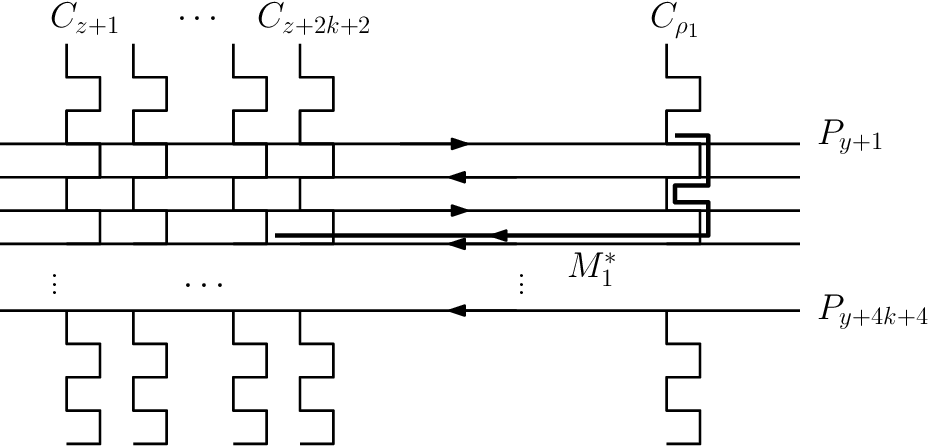}
         \caption{$r_1\in V(P_{y+1})$ and $\rho_1>z+2k+2$}
         \label{fig:listar1}
     \end{subfigure}
     \hfill
     \begin{subfigure}[b]{0.48\textwidth}
         \centering
         \includegraphics[width=\textwidth]{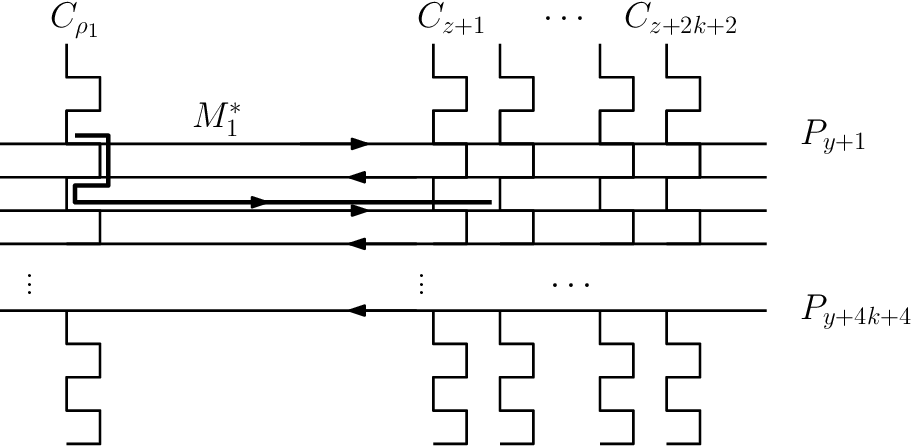}
         \caption{$r_1\in V(P_{y+1})$ and $\rho_1<z+1$}
         \label{fig:listar2}
     \end{subfigure} \hfill
      \vskip 0.3cm
         \begin{subfigure}[b]{0.48\textwidth}
         \centering
         \includegraphics[scale=0.6]{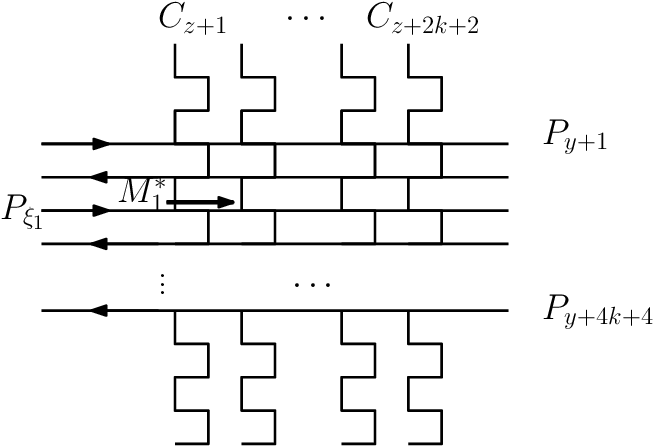}
         \caption{$r_1\in V(C_{z+1})$}
         \label{fig:listar3}
     \end{subfigure}\hfill
     \begin{subfigure}[b]{0.48\textwidth}
         \centering
         \includegraphics[scale=0.6]{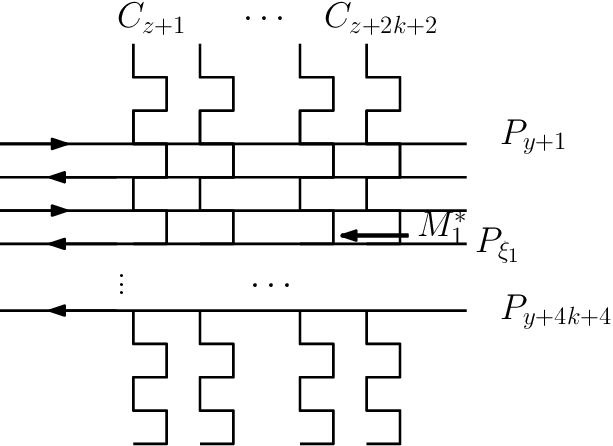}
         \caption{$r_1\in V(C_{z+2k+2})$}
         \label{fig:listar4}
     \end{subfigure}
        \caption{The construction of $M_1^*$.}
        \label{fig:listar}
\end{figure}

\begin{figure}[t]
     \centering
     \begin{subfigure}[b]{0.48\textwidth}
         \centering
         \includegraphics[width=\textwidth]{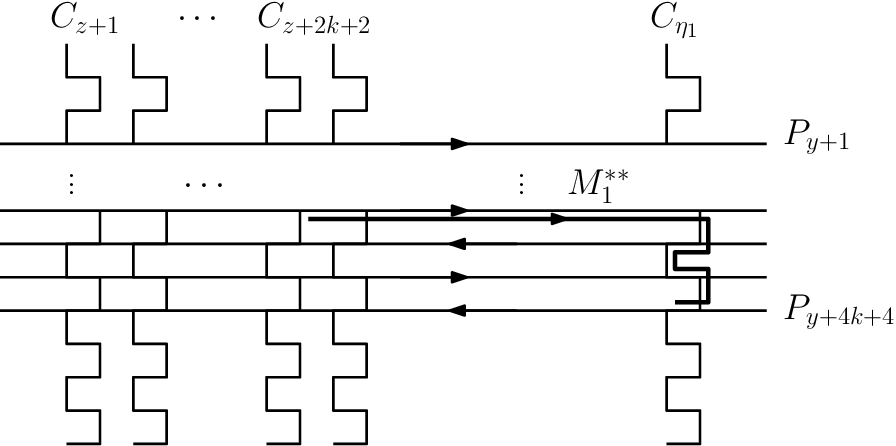}
         \caption{$s_1\in V(P_{y+4k+4})$ and $\eta_1>z+2k+2$}
         \label{fig:listarstar1}
     \end{subfigure}
     \hfill
     \begin{subfigure}[b]{0.48\textwidth}
         \centering
         \includegraphics[width=\textwidth]{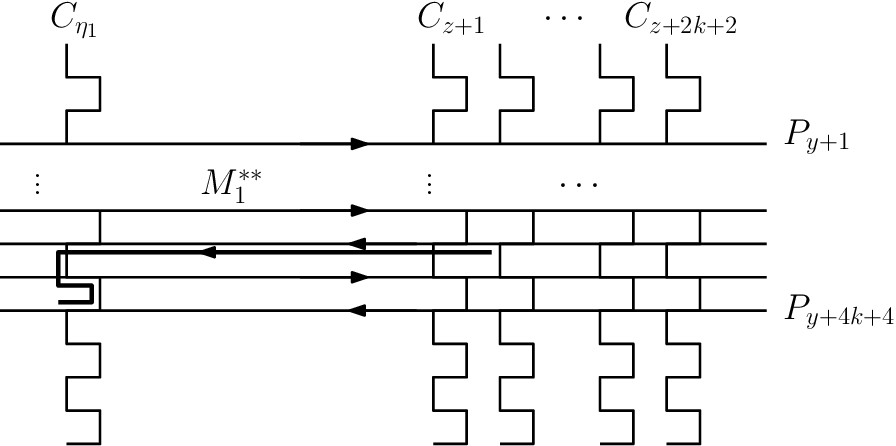}
         \caption{$s_1\in V(P_{y+4k+4})$ and $\eta_1<z+1$}
         \label{fig:listarstar2}
     \end{subfigure} \hfill
      \vskip 0.3cm
         \begin{subfigure}[b]{0.48\textwidth}
         \centering
         \includegraphics[scale=0.6]{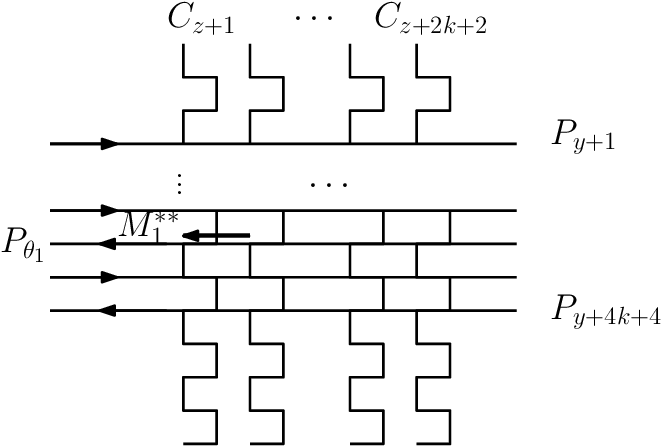}
         \caption{$s_1\in V(C_{z+1})$}
         \label{fig:listarstar3}
     \end{subfigure}\hfill
     \begin{subfigure}[b]{0.48\textwidth}
         \centering
         \includegraphics[scale=0.6]{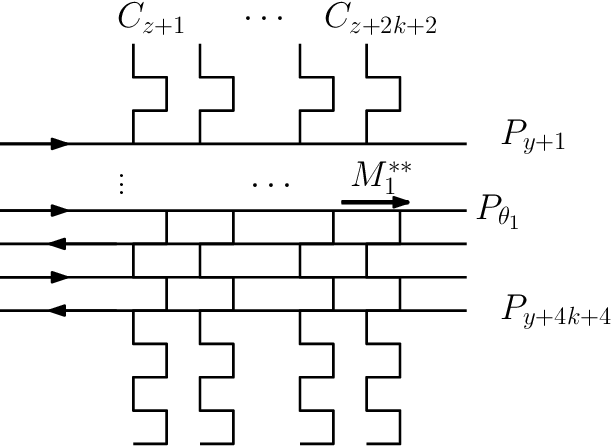}
         \caption{$s_1\in V(C_{z+2k+2})$}
         \label{fig:listarstar4}
     \end{subfigure}
        \caption{The construction of $M_1^{**}$.}
        \label{fig:listarstar}
\end{figure}

   \begin{itemize}
       \item Observe that $r_i$ is contained in $P_{y+1}\cup C_{z+1}\cup C_{z+2k+2}$.
       Let $C_{\rho_i}$ and $P_{\xi_i}$ be the column and row containing $r_i$ of $W$, respectively. See Figure~\ref{fig:listar} for illustrations.
       \begin{itemize}
           \item (Type 1. $r_i\in V(P_{y+1}).$)
           If $\rho_i>z+2k+2$, then let $M_i^*$ be the path starting at $r_i$, traversing to $L^{-1}(\rho_i, y+2i+2, 1)$ in $C_{\rho_i}$, and traversing to $L^{-1}(z+2k+2-i, y+2i+2, 1)$ in $P_{y+2i+2}$.
           If $\rho_i<z+1$, then let $M_i^*$ be the path starting at $r_i$, traversing to $L^{-1}(\rho_i, y+2i+1, 1)$ in $C_{\rho_i}$, and traversing to $L^{-1}(z+1+i, y+2i+1, 1)$ in $P_{y+2i+1}$.
           \item (Type 2. $r_i\in V(C_{z+1}).$)
            Let $M_i^*$ be the path starting at $r_i$ and traversing to $L^{-1}(z+1+i, \xi_i, 1)$ in $P_{\xi_i}$. 
            \item (Type 3. $r_i\in V(C_{z+2k+2})$.)
            Let $M_i^*$ be the path starting at $r_i$ and traversing to $L^{-1}(z+2k+2-i, \xi_i, 1)$ in $P_{\xi_i}$. 
       \end{itemize}
    \item The vertex $s_i$ is contained in $P_{y+4k+4}\cup C_{z+1}\cup C_{z+2k+2}$. Let $C_{\eta_i}$ and $P_{\theta_i}$ be the column and row containing $r_i$ of $W$, respectively. See Figure~\ref{fig:listarstar} for illustrations.
    \begin{itemize}
        \item (Type 1. $s_i\in V(P_{y+4k+4})$.) 
        If $\eta_i>z+2k+2$, then let $M_i^{**}$ be the path starting at $L^{-1}(z+2k+2-i, y+4k+3-2i, 2)$, traversing to $L^{-1}(\eta_i, y+4k+3-2i, 2)$ in $P_{y+4k+3-2i}$, and traversing to $L^{-1}(\eta_i, y+4k+4, 2)$ in $C_{\eta_i}$.
        If $\eta_i<z+1$, then let $M_i^{**}$ be the path starting at $L^{-1}(z+1+i, y+4k+4-2i, 2)$, traversing to $L^{-1}(\eta_i, y+4k+4-2i, 2)$ in $P_{y+4k+4-2i}$, and traversing to $L^{-1}(\eta_i, y+4k+4, 2)$ in $C_{\eta_i}$.
        \item (Type 2. $s_i\in V(C_{z+1}).)$ Let $M_i^{**}$ be the path starting at $L^{-1}(z+1+i, \theta_i, 2)$ and traversing to $L^{-1}(z+1, \theta_i, 2)$ in $P_{\theta_i}$. 
        \item (Type 3. $s_i\in V(C_{z+2k+2}).$) Let $M_i^{**}$ be the path starting at $L^{-1}(z+2k+2-i, \theta_i, 2)$ and traversing to $L^{-1}(z+2k+2, \theta_i, 2)$ in $P_{\theta_i}$.
   \end{itemize}
   \item Observe that the last vertex of $M_i^*$ and the first vertex of $M_i^{**}$ are contained in $C_{z+1+i}\cup C_{z+2k+2-i}$. Also, the subgraph $H_i$ obtained from $C_{z+1+i}\cup C_{z+2k+2-i}$ by adding the subpath of $P_{y+1+2i}$ from $C_{z+1+i}$ to $C_{z+2k+2-i}$ and the subpath of $P_{y+2+2i}$ from $C_{z+2k+2-i}$ to $C_{z+1+i}$ is strongly connected. 
   Let $M_i^{***}$ be a shortest path from the last vertex of $M^*$ to the first vertex of $M^{**}$ in $H_i$, and 
   let $Z_i^*:=M_i^{*}\cup M_i^{**}\cup M_i^{***}$. Clearly, $Z_i^*$ is a path from $r_i$ to $s_i$ in $\clos(W^*)$.
   \end{itemize}
   
 We claim that $\{Z_i^*:i\in [k]\}$ is a half-integral packing. 
 First observe that the set $\{M_i^*:i\in [k]\}$ is a half-integral packing. In fact, if $M_i^*$ intersects $M_j^*$ for some distinct $i,j\in [k]$, then  
 they are both paths of type 1, and either $\rho_i,\rho_j>z+2k+2$ or $\rho_i,\rho_j<z+1$. 
 But since  they traverse with pairwise distinct rows, no vertex can be shared by three paths in $\{M_i^*:i\in [k]\}$, 
 and furthermore, the possible intersection is not contained in the columns $C_{z+1}, \ldots, C_{z+2k+2}$.
 Similarly, the set $\{M_i^{**}:i\in [k]\}$ is a half-integral packing.
 Moreover, $\bigcup_{i\in [k]} M_i^*$ and $\bigcup_{i\in [k]} M_i^{**}$ are vertex-disjoint, because 
 we use rows $P_{y+3}, \ldots, P_{y+2k+2}$ for $M_i^*$ of type 1, 
 and rows $P_{y+2k+3}, \ldots, P_{y+4k+2}$ for $M_i^{**}$ of type 1, and 
 all paths of type 2 or 3 are pairwise vertex-disjoint ($r_i$ cannot be same as $s_j$ because of the directions).

 Thus, it is sufficient to consider nails contained in $C_{z+1}, \ldots, C_{z+2k+2}$, as paths in $\{M_i^{***}:i\in [k]\}$ do not use nails not contained in $C_{z+1}, \ldots, C_{z+2k+2}$.
 Suppose for contradiction that 
 there is a nail $v$ in $C_{z+1}, \ldots, C_{z+2k+2}$ 
 that is contained in some three paths in $\{M_i^*:i\in [k]\}\cup \{M_i^{**}:i\in [k]\}\cup \{M_i^{***}:i\in [k]\}$.
 Since paths in $\{M_i^*:i\in [k]\}\cup \{M_i^{**}:i\in [k]\}$ do not intersect on a nail in $C_{z+1}, \ldots, C_{z+2k+2}$, 
 $v$ is contained in two paths in $\{M_i^{***}:i\in [k]\}$, say $M_{i_1}^{***}$ and $M_{i_2}^{***}$.
 Since $\{H_i:i\in [k]\}$ is a half-integral packing, the other path should be a path in $\{M_i^*:i\in [k]\}\cup \{M_i^{**}:i\in [k]\}$.
 
 By the construction of $\{H_i:i\in [k]\}$, 
 $v$ is contained in one of the rows used by $M_{i_1}^{***}$ and $M_{i_2}^{***}$.
 But by the construction of $\{M_i^*:i\in [k]\}\cup \{M_i^{**}:i\in [k]\}$, 
 the other path should use the same row, and therefore, it has to have the same index as one of $i_1$ and $i_2$.
 Then the intersection vertex is contained in one path of $\{Z_i^*:i\in [k]\}$, contradicting the assumption that it is contained in three paths of $\{Z_i^*:i\in [k]\}$.
 We conclude that $\{Z_i^*:i\in [k]\}$ is a half-integral packing.

 Therefore, $\{Z_i\cup Z_i^*:i\in [k]\}$ is a half-integral packing of $k$ odd cycles, as required. 
\end{proof}

We now prove Theorem~\ref{thm:main}.
We recall that 
$\alpha_k$ is the minimum integer such that for every directed graph $G$ with $\nu_2(G)<k$, we have $\tau(G)\le \alpha_k$, if such an integer exists, and otherwise $\alpha_k$ is defined to be $\infty$.

\begin{proof}[Proof of Theorem~\ref{thm:main}]
	We prove by induction on $k$ that $\alpha_k\neq \infty$. We know $\alpha_1=0$. So, we may assume that $k>1$ and $\alpha_{k-1}\neq \infty$. 
	
	Let $f_{wall}$ be the function defined in Theorem~\ref{thm:KK2}, and let $g_{path}$ be the function defined in Proposition~\ref{prop:manywalks}. Let $r=2\alpha_{k-1}$. We set 
	\begin{itemize}
	    \item $f_3(k)=\max (k, r/4,  12(g_{path}(k)-1)+1)$,
	    \item $f_2(k)=\max ((2k+3)(6g_{path}(k)+1), 8f_3(k))$,
	    \item $f_1(k)=\max (r, f_{wall}(kf_2(k)))$,
	    \item $f(k)=\max (12f_1(k)(f_1(k)+1)+1, 24f_3(k)-4)$.
	\end{itemize}
	For convenience, let $w:=f_2(k)$.
	We show that for every directed graph $G$, if $\nu_2(G)<k$, then $\tau(G)\le f(k)$.
	
	Suppose for contradiction that $\nu_2(G)<k$ and $\tau(G)>f(k)$ for some directed graph $G$.
	Let $T$ be a minimum-size hitting set of odd cycles in $G$.
	By the assumption, $\abs{T}=\tau(G)>f(k).$
	Also, by Lemma~\ref{lem:welllinked}, 
	$T$ is $r$-externally-well-linked.

	Note that $2f_1(k)(f_1(k)+1)\ge r$ as $f_1(k)\ge r$.
	Since $\abs{T}>f(k)\ge 12f_1(k)(f_1(k)+1)+1$, 
	by Lemma~\ref{lem:largepathsystem}, 
	$G$ contains a well-linked set $A$ of size $f_1(k)$ such that
	\begin{itemize}
	    \item[($\ast$)] for every subset $Z$ of $T$ of size at least $\abs{T}/2$, 
	    there is a linkage of order $f_1(k)$ from $A$ to $Z$, and 
	    there is a linkage of order $f_1(k)$ from $Z$ to $A$.
	\end{itemize}

	Since $A$ is a well-linked set of size $f_1(k)\ge f_{wall}(kf_2(k))$, 
	by Theorem~\ref{thm:KK2}, 
	$G$ contains a cylindrical wall $W$ of order $kf_2(k)=kw$ such that for every set $F$ of $kw$ nails of $W$, 
	    there is a linkage of order $kw$ from $F$ to $A$, and 
	    there is a linkage of order $kw$ from $A$ to $F$.
     Let $C_1, \ldots, C_{kw}$ be the columns of $W$ and $P_1, \ldots, P_{2kw}$ be the rows of $W$.
	We consider the following $k$ vertex-disjoint subwalls of $W$.
    For each $j\in [k]$, 
    let $W_j$ be the subwall of $W$ consisting of columns $C_{w(j-1)+1}, \ldots, C_{wj}$ 
    and the minimal subpaths of rows $P_i$ with $i\in [2w]$ containing $C_{w(j-1)+1}\cap P_i$ and $C_{wj}\cap P_i$. 

    We claim that for each $j\in [k]$,
    \begin{itemize}
        \item[($\ast\ast$)] for every set $F$ of $w$ nails of $W_j$, 
	    there is a linkage of order $w$ from $F$ to $A$, and 
	    there is a linkage of order $w$ from $A$ to $F$.
    \end{itemize}
    Let $F$ be a set of $w$ nails of $W_j$. We choose a set $F'$ of $(k-1)w$ nails of $W$ that are not contained in $W_j$. We can choose such nails because there are $(k-1)w$ columns of $W$ that are not contained in $W_j$.
    By the property of $W$, 
    there is a linkage of order $kw$ from $F\cup F'$ to $A$, and there is a linkage of order $kw$ from $A$ to $F\cup F'$. 
    If we restrict paths whose endvertices are in $F$, then 
    we obtain a linkage of order $w$ from $F$ to $A$, 
    and a linkage of order $w$ from $A$ to $F$. Thus, the claim holds.

	If each of $W_1, \ldots, W_k$ contains an odd cycle, then we have $k$ vertex-disjoint odd cycles, contradicting the assumption that $\nu_2(G)<k$.
	Thus, one of $W_1, \ldots, W_k$, say $W'$, does not contain an odd cycle.

    Now, by $(\ast)$ and $(\ast\ast)$ and Lemma~\ref{lem:mintegral}, we have that 
	\begin{itemize}
	    \item[($\ast\ast\ast$)] for every subset $Z$ of $T$ of size at least $\abs{T}/2$ and every set $F$ of $w$ nails of $W$, 
	    there is a linkage of order at least $w/2$ from $F$ to $Z$, and 
	    there is a linkage of order at least $w/2$ from $Z$ to $F$.
	\end{itemize}
	Indeed, combining the linkage from $A$ to $Z$ and the linkage of order $w$ from $F$ to $A$, we obtain a half-integral linkage of order $w$ from $F$ to $Z$. Lemma~\ref{lem:mintegral} implies that there is a linkage of order at least $w/2$ from $F$ to $Z$. The other direction is similar.

	Since $W'$ has order $w$, $W'$ has $4w^2$ nails.
	Let $N$ be a set of $2w^2$ nails of $W'$ such that they are contained in the same part of the bipartition of $W'$.
	Now, we apply Lemma~\ref{lem:oddwalk2} for a tuple $(G, N, f_3(k))$.
	As $G$ has no half-integral packing of $k$ odd cycles and $f_3(k)\ge k$, $G$ contains either 
	\begin{itemize}
	    \item a half-integral packing $\mathcal{U}$ of $f_3(k)$ odd $N$-paths whose endvertices are pairwise disjoint, or
	    \item a set $Y$ of at most $4f_3(k)-1$ vertices such that $G-Y$ has no odd $N$-walks.
	\end{itemize} 
	
	Assume that the latter case happens. Observe that $f_2(k)\ge 8f_3(k)$, $4f_3(k)\ge r$, and $f(k)\ge 24f_3(k)-4$.
    We apply Proposition~\ref{prop:smallsep} with $(r,t,w)=(r, 4f_3(k), f_2(k))$. We can apply the proposition because of the property $(\ast\ast\ast)$.
    By Proposition~\ref{prop:smallsep}, 
	$G$ has a set of at most $3(4f_3(k)-1)$ vertices hitting all odd cycles. 
	It contradicts the fact that $\tau(G)>24f_3(k)-4\ge 12f_3(k)-3$.

    Thus, we may assume that the former case happens.
	Observe that $W'$ is a bipartite cylindrical wall of order \[w=f_2(k)\ge (2k+3)(6g_{path}(k)+1)\]
	and 
    $\mathcal{U}$ is a half-integral packing of
 \[f_3(k)\ge 12(g_{path}(k)-1)+1\] odd $N$-paths such that the endvertices of paths in $\mathcal{U}$ are disjoint.
	So, by Proposition~\ref{prop:manywalks}, 
	$\nu_2(G)\ge k$, a contradiction.
	
	We conclude that $\tau(G)\le f(k)$.

 \medskip
 (\textbf{Algorithmic part}.)
 	We now discuss how to turn this combinatorial result into a polynomial-time algorithm to find a half-integral packing of $k$ odd cycles or a hitting set of size at most $f(k)$, for fixed integer $k$.
	
	First by considering all sets $S$ of at most $f(k)$ vertices in $G$ and testing whether $G-S$ has no odd cycles, we can detect a hitting set of size at most $f(k)$ if one exists. Note that we can test in polynomial time whether a given directed graph has an odd cycle, as it is sufficient to test whether the underlying undirected graph of each strong component is bipartite. Therefore, we may assume that $G$ has no hitting set of odd cycles of size at most $f(k)$, that is, $\tau(G)>f(k)$. So, we want to find a half-integral packing of $k$ odd cycles.

	Note that we cannot guess the set $T$, as $\tau(G)$ may be much larger than $k$ (and must contain a half-integral packing of $k$ odd cycles). On the other hand, as $\tau(G)>f(k)$, there should be a well-linked set of size $f_1(k)$ as in the proof.
	We consider all sets $A$ of size $f_1(k)$ and test whether it is well-linked. As $k$ is a fixed integer, we can test in polynomial time  whether $A$ is well-linked, by repeatedly applying  Menger's theorem. 
	We construct the set 
	\[\mathcal{M}:=\{A\subseteq V(G): \abs{A}=f_1(k), A \text{ is well-linked in }G\}.\]
	
	As $f_1(k)\ge f_{wall}(kf_2(k))$, for each $A\in \mathcal{M}$, by applying Theorem~\ref{thm:KK2}, we obtain a cylindrical wall $W_A$ of order $kf_2(k)=kw$ such that 
	\begin{itemize}
	    \item for every set $F$ of $kw$ nails of $W_A$, 
	    there is a linkage of order $kw$ from $F$ to $A$, and 
	    there is a linkage of order $kw$ from $A$ to $F$.
	\end{itemize} Note that it runs in polynomial time for fixed $k$, and Campos et al.~\cite{Campos2020} recently discussed how to modify this into an FPT algorithm.
	By dividing $W_A$ into $k$ subwalls as in the proof, 
	we find either a half-integral packing of $k$ odd cycles or a bipartite cylindrical subwall $W_A'$ of order $w$ such that 
	 \begin{itemize}
        \item for every set $F$ of $w$ nails of $W_A'$, 
	    there is a linkage of order $w$ from $F$ to $A$, and 
	    there is a linkage of order $w$ from $A$ to $F$.
    \end{itemize}
    
    We choose a set $N_A$ of $w^2$ nails in $W_A'$ contained in the same part of the bipartition of $W_A'$.
    We apply Lemma~\ref{lem:oddwalk2} for the tuple $(G, N_A, f_3(k))$. Clearly, Lemma~\ref{lem:oddwalk2} can be simulated in polynomial time, as we only use Menger's theorem. If it outputs a half-integral packing of $k$ odd cycles, then we are done. So, we may assume that it outputs either 
		\begin{itemize}
	    \item a half-integral packing $\mathcal{U}_A$ of $f_3(k)$ odd $N_A$-paths whose endvertices are pairwise disjoint, or
	    \item a set $Y_A$ of at most $4f_3(k)-1$ vertices such that $G-Y_A$ has no odd $N_A$-walks.
	\end{itemize} 
	If this outputs $Y_A$ and the current set $A$ and the set $T$ satisfy the property $(\ast)$, then there is a hitting set of size at most $12f_3(k)-3\le f(k)$, which contradicts the assumption that $\tau(G)>f(k)$. 
	So, if the second outcome occurs, then it means that the current $A$ and $T$ do not satisfy $(\ast)$, and we skip this $A$.
	If it outputs $\mathcal{U}_A$, then following the proof of Proposition~\ref{prop:manywalks} we can obtain a half-integral packing of $k$ odd cycles in polynomial time.

	Since there exists a set $A\in \mathcal{M}$ satisfying $(\ast)$, by considering all sets $A$ in $\mathcal{M}$, we will either output a hitting set of size at most $f(k)$ or a half-integral packing of $k$ odd cycles. This concludes the algorithm.
	\end{proof}

\section{Discussion}\label{sec:discussion}

In this paper, we proved that a half-integral Erd\H{o}s-P\'osa theorem holds for directed odd cycles. We would like to ask whether a half-integral Erd\H{o}s-P\'osa theorem  holds for directed cycles of length $\ell$ modulo $m$ for other pairs of integers $\ell$ and $m$. Gollin et al.~\cite{GollinHKKO2021} proved that a half-integral Erd\H{o}s-P\'osa theorem holds for undirected cycles of length $\ell$ modulo $m$ for any pair of integers $\ell$ and $m\ge 2$.

\begin{question}\label{question1}
    For every pair of integers $\ell$ and $m\ge 2$, does there exist a function $f_{(\ell, m)}:\mathbb{N}\to \mathbb{R}$ satisfying that for every directed graph $G$ and every positive integer $k$, $G$ contains a half-integral packing of $k$ directed cycles of length $\ell$ modulo $m$, or 
a set of at most $f_{(\ell, m)}(k)$ vertices meeting all directed cycles of length $\ell$ modulo $m$? 
\end{question}

If Question~\ref{question1} is true for some pair $(\ell, m)$, then we further ask whether such a function still exists when we replace `a half-integral packing' with `an integral packing'.
Gollin et al.~\cite{GollinHKOY2022} characterized pairs of integers $(\ell, m)$ where an analogue of the Erd\H{o}s-P\'osa theorem holds for undirected cycles of length $\ell$ modulo $m$.

Gorsky, Kawarabayashi, Kreutzer, and Wiederrecht~\cite{GorskyKKW2024} proved that an analogue of the Erd\H{o}s-P\'osa theorem does not hold for directed even cycles, and a $1/4$-integral analogue of the Erd\H{o}s-P\'osa theorem holds for directed even cycles. They specifically asked whether a half-integral analogue of the Erd\H{o}s-P\'osa theorem holds for directed even cycles.

The function $f$ in Theorem~\ref{thm:main} is probably far from being optimal. We do not know any non-trivial lower bound for $f$ (other than $f(k)= \Omega(k)$). As far as we know, this is same for the function $g$ in Theorem~\ref{thm:Reed}. The original function $g$ in Theorem~\ref{thm:Reed} was exponential in $k$, but together with the polynomial grid theorem by Chuzhoy and Tan~\cite{ChuzhoyTan2021} and a result of Kawarabayashi, Thomas, and Wollan~\cite[Lemma 14.6]{KawaTW2020}, one can obtain a polynomial function $g$ for Theorem~\ref{thm:Reed}. We ask whether there is a polynomial function $f$ for Theorem~\ref{thm:main}. Also, finding non-trivial lower bounds for~$f$ and~$g$ would be interesting problems.

\begin{question}
Is there a polynomial function $f:\mathbb{N}\rightarrow \mathbb{R}$ such that for every directed graph $G$ and every positive integer $k$, 
$G$ contains a half-integral packing of $k$ directed odd cycles, or 
a set of at most $f(k)$ vertices meeting all directed odd cycles? 
\end{question}

\subsection*{Acknowledgements}
The authors would like to thank the anonymous referees for the careful reading of the manuscript and numerous suggestions that helped to improve the presentation.


\begin{thebibliography}{10}

\bibitem{AmiriKKW2016}
S.~A. Amiri, K.~Kawarabayashi, S.~Kreutzer, and P.~Wollan.
\newblock {The Erd\H{o}s-P\'osa property for directed graphs}.
\newblock {\em preprint}, arxiv.org/abs/1603.02504, 2016.

\bibitem{digraphbook}
J.~Bang-Jensen and G.~Gutin, editors.
\newblock {\em Classes of directed graphs}.
\newblock Springer Monographs in Mathematics. Springer, Cham, 2018.

\bibitem{BirmeleBR2007}
E.~Birmel{\'e}, J.~A. Bondy, and B.~A. Reed.
\newblock The {E}rd{\H o}s-{P}\'osa property for long circuits.
\newblock {\em Combinatorica}, 27(2):135--145, 2007.

\bibitem{BruhnJS2017}
H.~Bruhn, F.~Joos, and O.~Schaudt.
\newblock Long cycles through prescribed vertices have the {E}rd{\H
  o}s-{P}\'osa property.
\newblock {\em J. Graph Theory}, 87(3):275--284, 2018.

\bibitem{Campos2020}
V.~Campos, R.~Lopes, A.~K. Maia, and I.~Sau.
\newblock Adapting the directed grid theorem into an {FPT} algorithm.
\newblock {\em SIAM J. Discrete Math.}, 36(3):1887--1917, 2022.

\bibitem{ChuzhoyTan2021}
J.~Chuzhoy and Z.~Tan.
\newblock Towards tight(er) bounds for the excluded grid theorem.
\newblock {\em J. Combin. Theory Ser. B}, 146:219--265, 2021.

\bibitem{ErdosP1965}
P.~Erd{\H{o}}s and L.~P{\'o}sa.
\newblock On the independent circuits contained in a graph.
\newblock {\em Canad. J. Math.}, 17:347--352, 1965.

\bibitem{FioriniHRV2007}
S.~Fiorini, N.~Hardy, B.~Reed, and A.~Vetta.
\newblock Approximate min-max relations for odd cycles in planar graphs.
\newblock {\em Math. Program.}, 110(1, Ser. B):71--91, 2007.

\bibitem{FioriniH2014}
S.~Fiorini and A.~Herinckx.
\newblock A tighter {E}rd{\H o}s-{P}\'osa function for long cycles.
\newblock {\em J. Graph Theory}, 77(2):111--116, 2014.

\bibitem{GollinHKKO2021}
J.~P. Gollin, K.~Hendrey, K.~Kawarabayashi, O.~Kwon, and S.~Oum.
\newblock {A unified half-integral Erd\H{o}s-P\'osa theorem for cycles in
  graphs labelled by multiple abelian groups}.
\newblock {\em J. Lond. Math. Soc. (2)}, 109(1):Paper No. e12858, 35, 2024.

\bibitem{GollinHKOY2022}
J.~P. Gollin, K.~Hendrey, O.~Kwon, S.~Oum, and Y.~Yoo.
\newblock {A unified Erd\H{o}s-P\'osa theorem for cycles in graphs labelled by
  multiple abelian groups}.
\newblock {\em preprint}, arxiv.org/abs/2209.09488, 2022.

\bibitem{GorskyKKW2024}
M.~Gorsky, K.~Kawarabayashi, S.~Kreutzer, and S.~Wiederrecht.
\newblock Packing even directed circuits quarter-integrally.
\newblock In {\em S{TOC}'24---{P}roceedings of the 56th {A}nnual {ACM}
  {S}ymposium on {T}heory of {C}omputing}, pages 692--703. ACM, New York,
  [2024] \copyright 2024.

\bibitem{HuyneJW2017}
T.~Huynh, F.~Joos, and P.~Wollan.
\newblock {A unified Erd\H{o}s-P\'{o}sa theorem for constrained cycle}s.
\newblock {\em Combinatorica}, 39(1):91--133, 2019.

\bibitem{DBLP:journals/siamdm/KakimuraK12}
N.~Kakimura and K.~Kawarabayashi.
\newblock Packing directed circuits through prescribed vertices bounded
  fractionally.
\newblock {\em {SIAM} J. Discret. Math.}, 26(3):1121--1133, 2012.

\bibitem{KakimuraKM2011}
N.~Kakimura, K.~Kawarabayashi, and D.~Marx.
\newblock Packing cycles through prescribed vertices.
\newblock {\em J. Combin. Theory Ser. B}, 101(5):378--381, 2011.

\bibitem{soda13}
K.~Kawarabayashi, D.~Kr{\'{a}}l', M.~Krc{\'{a}}l, and S.~Kreutzer.
\newblock Packing directed cycles through a specified vertex set.
\newblock In {\em Proceedings of the Twenty-Fourth Annual {ACM-SIAM} Symposium
  on Discrete Algorithms, {SODA} 2013}, pages 365--377. {SIAM}, 2013.

\bibitem{KawaK2015J}
K.~Kawarabayashi and S.~Kreutzer.
\newblock {The directed grid theorem}.
\newblock {\em preprint}, arxiv.org/abs/1411.5681, 2014.

\bibitem{KawaK2015}
K.~Kawarabayashi and S.~Kreutzer.
\newblock The directed grid theorem.
\newblock In {\em S{TOC}'15---{P}roceedings of the 2015 {ACM} {S}ymposium on
  {T}heory of {C}omputing}, pages 655--664. ACM, New York, 2015.

\bibitem{KawaKKX2023}
K.~Kawarabayashi, S.~Kreutzer, O.~Kwon, and Q.~Xie.
\newblock {A half-integral {E}rd\H{o}s-{P}\'{o}sa theorem for directed odd
  cycles}.
\newblock In {\em Proceedings of the 2023 {A}nnual {ACM}-{SIAM} {S}ymposium on
  {D}iscrete {A}lgorithms ({SODA})}, pages 3043--3062. SIAM, Philadelphia, PA,
  2023.

\bibitem{KawaTW2020}
K.~Kawarabayashi, R.~Thomas, and P.~Wollan.
\newblock {Quickly excluding a non-planar graph}.
\newblock {\em preprint}, arxiv.org/abs/2010.12397, 2020.

\bibitem{KimK2020}
E.~J. Kim and O.~Kwon.
\newblock {Erd\H{o}s-{P}\'{o}sa property of chordless cycles and its
  applications}.
\newblock {\em J. Combin. Theory Ser. B}, 145:65--112, 2020.

\bibitem{KralSS2012}
D.~Kr\'{a}\v{l}, J.-S. Sereni, and L.~Stacho.
\newblock Min-max relations for odd cycles in planar graphs.
\newblock {\em SIAM J. Discrete Math.}, 26(3):884--895, 2012.

\bibitem{MasarikMPRS2019}
T.~Masa\v{r}\'{\i}k, I.~Muzi, M.~Pilipczuk, P.~{Rz\k{a}\.{z}ewski}, and
  M.~Sorge.
\newblock Packing directed cycles quarter- and half-integrally.
\newblock {\em Combinatorica}, 42:1409--1438, 2022.

\bibitem{Menger27}
K.~Menger.
\newblock Zur allgemeinen {K}urventheorie.
\newblock {\em Fund. Math.}, 10:96--115, 1927.

\bibitem{MoussetNSW2016}
F.~Mousset, A.~Noever, N.~{\v S}kori\'c, and F.~Weissenberger.
\newblock A tight {E}rd{\H o}s-{P}\'osa function for long cycles.
\newblock {\em J. Combin. Theory Ser. B}, 125:21--32, 2017.

\bibitem{PontecorviW2012}
M.~Pontecorvi and P.~Wollan.
\newblock Disjoint cycles intersecting a set of vertices.
\newblock {\em J. Combin. Theory Ser. B}, 102(5):1134--1141, 2012.

\bibitem{RaymondT16}
J.-F. Raymond and D.~M. Thilikos.
\newblock Recent techniques and results on the {E}rd{\H o}s-{P}\'osa property.
\newblock {\em Discrete Appl. Math.}, 231:25--43, 2017.

\bibitem{Reed1999}
B.~Reed.
\newblock Mangoes and blueberries.
\newblock {\em Combinatorica}, 19(2):267--296, 1999.

\bibitem{ReedRST1996}
B.~Reed, N.~Robertson, P.~Seymour, and R.~Thomas.
\newblock Packing directed circuits.
\newblock {\em Combinatorica}, 16(4):535--554, 1996.

\bibitem{Reed99}
B.~A. Reed.
\newblock Introducing directed tree width.
\newblock {\em Electron. Notes Discret. Math.}, 3:222--229, 1999.

\bibitem{RobertsonS1986}
N.~Robertson and P.~D. Seymour.
\newblock Graph minors. {V}. {E}xcluding a planar graph.
\newblock {\em J. Combin. Theory Ser. B}, 41(1):92--114, 1986.

\bibitem{Thomassen1988}
C.~Thomassen.
\newblock On the presence of disjoint subgraphs of a specified type.
\newblock {\em J. Graph Theory}, 12(1):101--111, 1988.

\bibitem{BatenHJR2020}
W.~C. van Batenburg, T.~Huynh, G.~Joret, and J.-F. Raymond.
\newblock {A tight Erd\H{o}s-P\'{o}sa function for planar minors}.
\newblock {\em Advances in Combinatorics}, 2:33 pages, 2019.

\bibitem{younger}
D.~H. Younger.
\newblock Graphs with interlinked directed circuits.
\newblock {\em Proceedings of the Midwest Symposium on Circuit Theory}, 2:XVI
  2.1 -- XVI 2.7, 1973.

\end{thebibliography}
    \end{document}